\documentclass[11pt]{article} 
\usepackage[margin=1in]{geometry} 
\usepackage{amsmath, amssymb, amsthm} 
\usepackage{mathrsfs}
\usepackage{bbm} 
\usepackage{hyperref} 
\usepackage[dvipsnames]{xcolor} 
\usepackage{dsfont} 
\usepackage{natbib}
\usepackage{tikz} 
\usepackage{enumitem}
\usepackage{cleveref}
\usepackage{thmtools, thm-restate}
\usepackage{subcaption} 

\usepackage[textwidth=3cm]{todonotes}

\newcommand{\PMmx}{\ensuremath{{\mathbb{P}\left(M=m \mid X=x\right)}}}
\newcommand{\PMmxm}{\ensuremath{{\mathbb{P}\left(M=m \mid X^{(m)}=x^{(m)}\right)}}}
\newcommand{\PMzerox}{\ensuremath{{\mathbb{P}\left(M=0 \mid X=x\right)}}}

\newcommand{\E}{\ensuremath{{\mathbb E}}}
\newcommand{\R}{\ensuremath{{\mathbb R}}}
\newcommand{\Prob}{\ensuremath{{\mathbb P}}}

\newcommand{\data}{\mathcal{S}_n}

\newcommand{\Pjointn}{\ensuremath{\mathbb{P}_{\theta^*}^n}}
\newcommand{\Pthetajointn}{\ensuremath{\mathbb{P}_{\theta}^n}}
\newcommand{\Pjoint}{\ensuremath{\mathbb{P}_{\theta^*}}}
\newcommand{\Pthetajoint}{\ensuremath{\mathbb{P}_{\theta}}}
\newcommand{\Gauss}[2]{\mathcal{N}\left(#1,#2\right)}

\newcommand{\KLMAR}{\ensuremath{\widetilde{\textnormal{KL}}}}
\newcommand{\KLMARsq}{\ensuremath{\widetilde{\textnormal{V}}}}
\newcommand{\KL}{\ensuremath{{\textnormal{KL}}}}
\newcommand{\He}{\ensuremath{\textnormal{H}}}
\newcommand{\HeMAR}{\ensuremath{\widetilde{\textnormal{H}}}}
\newcommand{\Af}{\ensuremath{\rho}_{1/2}}
\newcommand{\AfMAR}{\ensuremath{\tilde{\rho}_{1/2}}}
\newcommand{\supp}{\ensuremath{{\textnormal{supp}}}}
\newcommand{\logm}{\ensuremath{\log_{-}}}

\usepackage{algorithm}
\usepackage{algpseudocode}

\newcommand{\CI}{C_1}
\newcommand{\CII}{C_2}
\newcommand{\CIII}{C_3}
\newcommand{\CIV}{C_4}

\newcommand{\NI}{N_1}
\newcommand{\NII}{N_2}

\title{Asymptotics of Nonparametric Estimation under general non-monotone MAR missingness: A Bayesian Approach}

\author{Jeffrey Näf$^1$, Badr-Eddine Chérief-Abdellatif$^2$, \vspace{0.2cm} \\
        $ ^1$Research Institute for Statistics and Information Science,\\
        University of Geneva \\
        $ ^2$CNRS, LPSM, Sorbonne Université, Université Paris Cité \vspace{0.1cm} \\
}
\date{} 

\newtheorem{thm}{Theorem}[section] 
\newtheorem{asm}{Assumption}[section] 
\newtheorem{prop}[thm]{Proposition} 
\newtheorem{lemma}[thm]{Lemma} 
\newtheorem{cor}[thm]{Corollary} 
\newtheorem{dfn}{Definition}[section]

\allowdisplaybreaks 

\begin{document}

\maketitle

\begin{abstract}

Missing values are ubiquitous in statistical practice, with potentially detrimental consequences for any statistical analysis. As such, a wealth of methods and theoretical results have been developed in the last decades. However, many questions remain open, in particular in the case of general non-monotone missing at random (MAR), where nonparametric results are still lacking. In this paper, we extend nonparametric Bayesian theory to this MAR setting. We introduce a general theorem of posterior contraction under MAR and an additional positivity condition and apply this result to density estimation as well as regression problems. In particular, we show that, despite the missing values, the complete-data density can be estimated with the minimax posterior contraction rate up to logarithmic factors. To the best of our knowledge, this is the first nonparametric result showing that the complete-data distribution can be consistently estimated under Rubin's MAR definition. As a consequence, we obtain an algorithm that takes incomplete data and returns a sample from a consistent estimate of the complete-data distribution.

\end{abstract}


\section{Introduction}

Missing data are a prevalent issue in modern data science applications and an active area of research. In the presence of missing data, one no longer observes complete data points, but only partial values, along with the positions of the observed entries. In general, this requires introducing a model for the missing data mechanism (MDM), which governs which values are missing. A classical and effective approach, introduced by Rubin~\cite{Rubin1976}, is to model the MDM using conditional distributions that satisfy the \emph{Missing At Random} (MAR) property - meaning that the probability of missingness depends only on the observed values - and to parameterize the MDM separately from the data model. Under this modeling choice, the MDM becomes \emph{ignorable} in the sense that maximizing the likelihood of both the observed values and their positions over the joint model yields exactly the same estimator for the data parameter $\theta$ as maximizing the marginal likelihood of the observed values alone. This principle, which also extends to Bayesian inference when prior independence is assumed between the data and MDM parameters, has greatly simplified inference and computations in the field of missing data analysis.

\vspace{0.2cm}
Despite the intuitive appeal and widespread adoption of the “ignorability” principle, which allows researchers and practitioners to estimate $\theta$ without modeling the potentially complex missing mechanism, its theoretical foundations remained unsettled for decades. Indeed, while a MAR MDM model renders the missing mechanism ignorable in the definition of the maximum likelihood estimator (MLE) $\widetilde{\theta}_n$, this does not in itself guarantee the statistical validity of inference based on the so-called ignorable likelihood. Rubin's influential formulation was persuasive and has shaped much of the applied and methodological literature on missing data; yet rigorous mathematical justifications were largely absent until relatively recently. As emphasized by~\cite{Takai2013}, many foundational texts - including the famous monograph ~\cite{little2019statistical} - present inference based on the incomplete-data likelihood as if it naturally inherits the large-sample properties of complete-data inference, without providing formal proofs of consistency or asymptotic normality for the resulting estimators. In a remarkable contribution, \cite{Takai2013} addressed this gap by proving that, under standard regularity conditions, the MLE ignoring the MDM $\widetilde{\theta}_n$ is indeed consistent and asymptotically normal, provided the true missing data mechanism is itself MAR. In this sense, the result of \cite{Takai2013} provides a theoretical validation of the long-standing claim that the MAR nature of the true missing data mechanism is the key condition under which ignorable likelihood-based inference is reliable.

\vspace{0.2cm}
While this settles the case of the MLE $\widetilde{\theta}_n$ for regular parametric models, several fundamental questions remain open:
\begin{itemize}
    \item First, does Bayesian inference, under comparable conditions, enjoy similar guarantees ?
    \item Second, can the statistical validity of inference based on the ignorable likelihood extend beyond the parametric setting ?
\end{itemize}
These questions motivate the present work. We show that the rich Bayesian theory in conjunction with the natural use of the Kullback-Leibler (KL) divergence allows us to obtain general nonparametric convergence results under MAR missingness. We apply this to nonparametric regression with missing covariates and missing response, respectively, as well as to nonparametric density estimation under general non-monotone missingness. In particular, we show that it is possible to nonparametrically estimate the complete-data density under general non-monotone MAR missingness and a natural positivity assumption. We thus naturally arrive at a method that is able to use MAR missing data to produce new samples from the (estimated) complete-data distribution. Though the missingness literature is vast, these appear to be the first nonparametric consistency results for general, non-monotone MAR. As the concept of MAR is nearly 50 years old and is seen by some as a solved issue, this might seem surprising. Indeed, mirroring the discussion of MAR in the MLE case, papers often claim that "X is valid under MAR". However, as discussed in \cite{whatismeant, näf2024goodimputationmarmissingness} and others, this impression likely stems, in part, from confusions about the MAR condition itself. Theoretical results have indeed been obtained under stronger missingness conditions, such as the probability of missingness depending on an always observed subset of the data. Moreover, a wealth of theoretical results emerged in the case of structured problems, such as specific regression problems with missingness, often in conjunction with such simplified missing mechanisms; see, for instance, \cite{EmpiricalLik1, EmpiricalLik2, EmpiricalLik3, Responsequantileimputation, ELW2023, imputationpoweredinference}. However, to the best of our knowledge, consistency results under general MAR outside of (parametric) MLE have not been developed before. In fact, even trying to extend the ignoring principle to M-Estimation yields inconsistent estimators in general, see, e.g., \cite{Mestimatormissingvalues}. 

\vspace{0.2cm}
The remainder of the paper is organized as follows: In Section \ref{sec_Background}, we present a detailed background on missing values, the MAR condition and ignorability, and introduce our notation. After introducing the MAR condition and relevant notation, we discuss the related literature in more detail and outline our contributions in Section \ref{sec_literature_lontributions}. Section \ref{sec_contractionresults} then presents the general posterior contraction results under MAR missingness, while Section \ref{sec_examples} applies these results to density estimation and nonparametric regression on $\R^d$. Finally, Section \ref{sec_empirical} provides a small simulation study, and Section \ref{sec_conclusion} concludes. 



\begin{figure*}[h!]
    \centering
\begin{tikzpicture}


\node at (0,0) {$ \begin{pmatrix}
x_{1,1} & x_{1,2} & x_{1,3} \\
\textrm{NA} & x_{2,2} & x_{2,3} \\
\textrm{NA} & \textrm{NA} & x_{3,3}
\end{pmatrix}$};

\node at (4,0) {$\begin{pmatrix}
x_{1,1} & x_{1,2} & x_{1,3} \\
 \textrm{NA} & x_{1,2} & x_{2,3} \\
x_{3,1} & \textrm{NA} & x_{3,3}
\end{pmatrix}$};


\node at (8,0) {$\begin{pmatrix}
x_{1,1} & x_{1,2} & x_{1,3} \\
 x_{1,2} &\textrm{NA} & x_{2,3} \\
\textrm{NA} & x_{3,2} &\textrm{NA} 
\end{pmatrix}$};

\end{tikzpicture}
    \caption{Three Data matrices with missing values, each with three different patterns. Each contains the fully observed pattern $M=0$ and does not contain the completely unobserved pattern $M\neq \mathbbm{1}$.}
    \label{fig:illustrationfocond}
\end{figure*}

\section{Background and Notation}\label{sec_Background}

Let $X_1, \dots, X_n$ be random vectors taking values in a measurable space $(\mathcal{X}, \mathcal{A})$, $\mathcal{X} \subset \mathbb{R}^d$. The $X_i$'s are assumed to be i.i.d.\ (independent and identically distributed) from a distribution $P_{\theta^*}$ belonging to a statistical model $\{P_\theta\}_{\theta \in \Theta}$. The parameter space $\Theta$ is equipped with a suitable $\sigma$-algebra and with a semi-metric $d$, and may be either finite- or infinite-dimensional. This general formulation encompasses, in particular, high-dimensional and nonparametric settings, such as models where $\Theta$ is a class of probability density functions with respect to a given reference measure. Furthermore, we assume that each $P_\theta$ admits a density $p_\theta$ with respect to Lebesgue's measure on $\mathcal{X}$ and that the mapping $(\theta, x) \mapsto p_\theta(x)$ is measurable.

\vspace{0.2cm}
Unfortunately, we do not observe the $X_i$'s directly. Instead, for each $i \in \{1, \dots, n\}$, we only observe a subset of its $d$ components. Let $M_i \in \{0,1\}^d$ be a random binary mask indicating which coordinates of $X_i$ are observed:
\[
(M_i)_j = 
\begin{cases}
0 & \text{if the } j\text{-th component of } X_i \text{ is observed}, \\
1 & \text{if the } j\text{-th component of } X_i \text{ is missing}.
\end{cases}
\]
We assume that the pairs $(X_i, M_i)$ are i.i.d.\ with a joint distribution $\mathbb{P}_{(X,M)}$ on $\mathcal{X} \times \{0,1\}^d$ which is assumed to be absolutely continuous with respect to the product of Lebesgue’s and counting measures. We denote by $X_i^{(M_i)}$ the restriction of $X_i$ to its observed components, that is, the subvector containing only the entries for which $(M_i)_j = 0$. The available data thus consists of the partially observed vectors along with their positions 
$$
\data = \left\{ \left(X_1^{(M_1)},M_1\right),  \dots, \left(X_n^{(M_n)},M_n\right) \right\} \, .
$$
Our goal is to estimate the true parameter $\theta^* \in \Theta$ with only access to $\data$.

\subsection{Missingness at Random}
We focus in this paper on a specific family of missingness mechanisms usually referred to as \emph{Missing at Random}, which is encoded in the following assumption:

\begin{asm}[MAR]
\label{asm_true_MDM}
    The true conditional distribution of $M$ given $X$ is \emph{Missing at Random} (MAR), meaning that for almost any $x\in\mathcal{X}$, the probability mass function of $M$ given $X=x$ only depends on its observed components: for almost any $x\in\mathcal{X}$, for any $m\in\{0,1\}^d$, 
    $$
   \PMmx = \PMmxm .
    $$
\end{asm}
The MAR property used above, formally stating that the missingness mechanism (i.e.\ the conditional distribution of $M$ given $X$) does not depend on the missing values themselves (given the observed ones), is a particular instance (Missing Always at Random) of the several variants of Missingness at Random that exist in the literature, see, e.g., \cite{whatismeant3, näf2024goodimputationmarmissingness} and the literature therein. We note that Assumption \ref{asm_true_MDM} is essentially the population version of the original MAR version of \cite{Rubin1976}. Crucially, it is the weakest MAR assumption compared to the alternatives used in the literature. For instance, it is often assumed that $\Prob\left(M=m\mid X\right)$ only depends on a set of fully observed variables, a much stronger assumption (see, e.g., the discussion in \cite{näf2024goodimputationmarmissingness}). 

\vspace{0.2cm}
We do not assume anything on the structure of the patterns $M$, other than that the completely empty pattern $M=\mathbbm{1}$ (where $\mathbbm{1}=(1,\ldots,1)$) has probability zero:

\begin{asm}[Disregarding the empty pattern]
\label{asm_no_empty_MDM}
    We assume that we almost never observe nothing, that is $\Prob(M=\mathbbm{1})=0$.
\end{asm}

In particular, missingness can be non-monotone as outlined in Figure \ref{fig:illustrationfocond}. While the first example contains monotone missingness, the second example already breaks monotonicity, having $X_2$ missing in the second pattern and in the third pattern $X_3$. Still in the simplified case where missingness $\Prob\left(M=m\mid X\right)$ depends only on $X_3$, it is still relatively straightforward to learn $\Prob\left(M=m\mid X\right)$ and apply reweighting estimators. However, under MAR, missingness might depend on $X_2$ in the second pattern and $X_1$ in the third pattern. This example is discussed in Section \ref{Sec_Motivating}. The situation is even more complicated in the third example, where in addition to non-monotonicity there is no fully observed variable anymore.

\vspace{0.2cm}
We also introduce a technical assumption, standard in the literature, that requires the so-called \emph{propensity score}, defined as the function $x \mapsto \PMzerox$ giving the conditional probability of being fully observed, to be bounded away from zero:
\begin{asm}[Positivity of the propensity score]
\label{asm_positive_MDM}
    The propensity score $x\mapsto \PMzerox$ is bounded away from zero: there exists some constant $\delta \geq 0$ such that for almost any $x$, we have $\PMzerox >  \delta$.
\end{asm}

For $\delta > 0$, this is a rather standard assumption for the theoretical analysis of missing values (see, e.g., \cite{MARinverseweighting, Malinsky2022}). For some of our results, it will also be possible to take $\delta=0$, though the rates we derive explicitly depend on $1/\delta$, which is understood to be infinite if $\delta=0$.

\vspace{0.2cm}
We end this subsection with a useful notation: we denote \( \Prob_{\theta} \) the joint distribution of \( (X,M) \) when \( X \sim P_\theta \) and \( M \) follows the unknown conditional missingness mechanism. In particular, $\Prob_{(X,M)} = \Pjoint$. Of course, all distributions $\Prob_{\theta}$ are unknown since the true conditional missingness mechanism is itself unknown.


\subsection{Ignorability of likelihood-based inference under MAR}

Likelihood-based inference can be applied in the presence of missing values. Since the pair $(X,M)$ is random, the central idea is to model its joint distribution by combining the complete-data model $\{P_\theta\}_\theta$ for $X$ with a model for the missingness mechanism parameterized by some $\phi$. Given the partially observed dataset $\data$, the maximum likelihood estimator is obtained by maximizing the likelihood of $\data$ with respect to $(\theta,\phi)$.

\vspace{0.2cm}
A fundamental result in the missing-data literature, known as the \emph{ignorability} property of the MLE and established by \cite{Rubin1976}, shows that inference based on the estimator
\[
\widetilde{\theta}_n = \arg\max_{\theta\in\Theta} \sum_{i=1}^{n} \log p_{\theta}^{(M_i)}\!\left(X_i^{(M_i)}\right) \, ,
\]
where $p_\theta^{(m)}(x)$ denotes the marginal density of $x$ corresponding to the components observed under mask $m\in\{0,1\}^d$, is equivalent to the maximum likelihood estimator computed from the entire observed dataset $\data$ under the MAR (Assumption~\ref{asm_true_MDM}), provided that $\theta$ and $\phi$ are distinct, meaning that they do not share any parameters. This maximum (ignorable) likelihood estimator (referred to simply as the MLE from now on) optimizes over $\theta$ while ignoring the missingness mechanism. It is particularly convenient in practice, as it avoids the often difficult task of specifying a model for the missingness mechanism.

\vspace{0.2cm}
Similarly, under the MAR Assumption~\ref{asm_true_MDM}, the distinctness condition on $\theta$ and $\phi$, and assuming independent priors for $\theta$ and $\phi$, Bayesian inference for $\theta$ is also ignorable. The posterior density over $\theta$ (again ignoring the missingness mechanism) is given by
\[
    \pi(\theta \mid \data) \propto \prod_{i=1}^{n} p_\theta^{(M_i)}\!\left(X_i^{(M_i)}\right) \, \pi(\theta),
\]
where $\pi(\cdot)$ denotes the density of a prior distribution $\Pi$ with respect to some reference measure. We denote by $\Pi(\cdot \mid \data)$ the corresponding posterior distribution.

\subsection{Notation}

We now introduce and summarize the notation used throughout the paper.

\begin{itemize}
\item We assume to observe masked i.i.d. data along with their positions, $$\data = \left\{ \left(X_1^{(M_1)},M_1\right),  \dots, \left(X_n^{(M_n)},M_n\right) \right\}.$$
   \item $\Pjoint$ is the joint distribution of $(X,M)$. We denote the conditional distribution of $M \mid X$ as $\PMmx$ for $m \in \{0,1\}^d$ and almost all $x \in \mathcal{X}$, which is well defined on $\mathcal{X} \subset \R^d$. We assume that $\Pjoint$ is absolutely continuous with respect to the product of Lebesgue's and counting measures, with joint density $(x,m) \mapsto  p_{\theta^*}(x) \PMmx$.
   \item We assume a (semi-)metric space $(\Theta, d)$, inducing a model class $(P_{\theta})_{\theta \in \Theta}$ for the distribution of $X$. We assume $P_{\theta^*} \in (P_{\theta})_{\theta \in \Theta}$ (i.e., the well-specified case). In addition, $(\Theta, d)$ also induces a model class $(\Pthetajoint)_{\theta \in \Theta}$ of $(X,M)$, whereby $\Pthetajoint$ is the distribution induced by the density $(x,m) \mapsto  p_{\theta}(x) \PMmx$.
   \item For $m \in \{0,1\}^d$, $X^{(m)}$ is the subvector of $X$ corresponding to the variables such that $m_j=0$, while $X^{(-m)}$ is the subvector of $X$ corresponding to the variables such that $m_j=1$. We denote the corresponding marginal densities by $p_{\theta}^{(m)}/p_{\theta}^{(-m)}$.
   \item We further assume a prior distribution $\Pi$ on $\Theta$, equipped with the Borel $\sigma$-algebra induced by $d$. We then denote by $\Pi(\cdot \mid \data)$ the posterior distribution,
   \begin{align*}
       \Pi(B \mid \data) = \frac{\int_B \prod_{i=1}^{n} p_{\theta}^{(M_i)}(X_i^{(M_i)}) d\Pi(\theta)}{\int \prod_{i=1}^{n} p_{\theta}^{(M_i)}(X_i^{(M_i)})d\Pi(\theta)},
   \end{align*}
   for all $B$ in the Borel $\sigma$-algebra induced by $d$, which defines a probability measure on $\Theta$. 
    \item For complete data, we denote by $\KL(P_{\theta_1} \|P_{\theta_2})$, the KL divergence between $P_{\theta_1},P_{\theta_2}$,
    \begin{align*}
        \KL(P_{\theta_1} \| P_{\theta_2})=\E_{X \sim P_{\theta_1}}\left[ \log \left( \frac{p_{\theta_1}(X)}{p_{\theta_2}(X)} \right) \right].
    \end{align*}
    We note that this expectation is always defined and nonnegative, though $\KL(P_{\theta_1} \| P_{\theta_2})=\infty$, if $P_{\theta_1}(\{p_{\theta_2}(X)=0\}) > 0$.
    \item For complete data, the Hellinger distance between $P_{\theta_1}, P_{\theta_2}$ is given as:
\[
\He^2(p_{\theta_1}, p_{\theta_2}) = \int \left( \sqrt{p_{\theta_1}}(x)- \sqrt{p_{\theta_2}}(x)\right)^2 \mathrm{d}x.
\]
\item For a subset $\Theta_0 \subset \Theta$, we let $N(\varepsilon_n, \Theta_0, d)$ be the usual covering number, i.e., the minimal covering of $\Theta_0$ by open $d-$ balls.
\item An important metric $d$ on $\Theta$ will be the one induced by the Hellinger distance on $(P_{\theta})_{\theta \in \Theta}$, that is,
\[
d_H(\theta_1, \theta_2):= \He(p_{\theta_1},p_{\theta_2}).
\]
\end{itemize}


\section{Problem Statement and Contributions}\label{sec_literature_lontributions}

In this section, we first dive deeper into the related literature and then discuss our contributions.

\subsection{Related Literature}



The term MAR has been a frequent cause of confusion in the literature, as indicated by the number of papers that simply discuss the definitions of MAR, e.g., \cite{whatismeant, whatismeant3, näf2024goodimputationmarmissingness}. It is often claimed in the literature that \textit{"ignorable likelihood-based inference is valid under MAR"}. However, such statements are often ambiguous. In fact, frequentist validity, in the sense of consistency, has not been formally established in general models and has only relatively recently been addressed for regular parametric complete-data models. What has been clearly established by \cite{Rubin1976} is the equivalence of likelihood-based inference using the joint likelihood, modeling $\data$, and the ignorable likelihood, provided the \textit{model} for the missing data mechanism is MAR (and the parameters are distinct) - regardless of the nature of the true missing mechanism. However, whether this approach remains statistically valid when the \emph{true} missing data mechanism is MAR (as formulated in Assumption~\ref{asm_true_MDM}) is a different and, until recently, unresolved question.

\vspace{0.2cm}
Conversely, while there is now rich parametric and nonparametric theory developing for the handling of missing values, guarantees for the general non-monotone MAR case are notably sparse. This is not entirely surprising. As discussed in detail in \cite{näf2024goodimputationmarmissingness}, the MAR in Assumption \ref{asm_true_MDM} is rather complicated to handle. In particular, in Section \ref{Sec_Motivating} we demonstrate that, already in three dimensions, complex distribution shifts can arise when changing from one pattern to another. While there are interesting results for the parametric case in MLE estimation (\cite{Takai2013}) and imputation (\cite{RobinsImputationTheory, Guan2024}), we are not aware of general nonparametric results in this case. For instance, in the context of M-estimation, \cite{Mestimatormissingvalues} show that a simple ignoring estimator is no longer guaranteed to be consistent under MAR. An intuitive remedy are approaches that weight the data using the inverse of an estimate of the conditional probability $\PMmx$, giving rise to inverse probability weighting (IPW) estimators. However, under non-monotone MAR, this approach is not straightforward. For instance, \cite{MARinverseweighting} discusses the difficulty of IPW estimators in this setting and proposes a parametric model of missingness probabilities that allows weighting of the data to obtain consistency. However, this parametric form appears rather limiting, and IPW-based estimators that require an estimate of the pattern probabilities can be difficult to handle in our setting, where it might be possible to observe a pattern only infrequently. In fact, even in simpler settings IPW is known to suffer from instability when propensities are close to zero (\cite{kang2007demystifying, han2019quantile}). As a consequence, the theory of IPW estimators seems to have been developed mainly in the context of monotone missingness (see, e.g., \cite{Seaman2018-IPW}). A notable exception to this, and one of the very few papers that is able to provide asymptotic guarantees for M-estimation under general non-monotone patterns outside of MCAR, is the seminal paper \cite{Malinsky2022}. However, they are studying the ``no-self-censoring'' mechanism that is fundamentally different from MAR. Moreover, their approach again requires a rather complex reweighting by estimated propensity scores. In contrast, the ignoring approach we adopt is straightforward and completely circumvents the need for any weighting. In another important recent approach \cite{nonparametricpmm} estimate the joint distribution $\Pjoint$ of a multivariate sample under missing values, using kernel density estimators and identification conditions. They provide general consistency and even asymptotic normality of estimators derived from this distributional estimate. This is close in spirit to the density estimator we obtain as a natural application of our theory. However, they again mainly focus on monotone missingness. While they also discuss the possibility of non-monotone missingness, their ``donor-based'' approach requires rather complex identification conditions that do not hold under non-monotone MAR. In particular, their method could not be used in the example in Section \ref{Sec_Motivating}. This difficulty arises because the approach requires to learn (conditional) densities for one pattern, from other patterns, or ``donors'', and the procedure needs to be adapted to the assumed condition. Thus, while their approach is promising, it also demonstrates the difficulty of non-monotone missingness. In contrast, our approach is straightforward and valid even under MAR. On a technical note, just as in the case of complete data, the Bayesian density estimation approach also has the advantage of being able to adapt to higher rates of smoothness compared to the kernel density estimator, see, e.g., \cite[Chapter 9]{Fundamentals}. Finally, \cite{chen2022pattern} generalizes several prior identification restrictions, such as the donor-based restrictions of \cite{nonparametricpmm}, into a single framework. In this approach, the so-called "pattern graphs" encode the missingness assumptions, which can even be missing not at random (MNAR). This elegant construction allows one to build IPW and even imputation-based (M-) estimators that are consistent and asymptotically Gaussian. While this is another viable alternative to the ignoring approach we pursue, it comes at the price of assuming a specific graph, and, as before, it is unclear how to apply this to general MAR situations, such as the example in Section \ref{Sec_Motivating}.




\subsection{Contributions}





In regular parametric models with complete data, the frequentist validity of Bayesian procedures is well established. It is well known that when the MLE is consistent and asymptotically normal, the posterior distribution typically concentrates around the true value and satisfies a Bernstein–von Mises (BvM) theorem, under mild conditions on the prior. These properties can be informally viewed as the Bayesian counterparts of consistency and asymptotic normality, ensuring that the posterior behaves like a good frequentist estimator. However, in the presence of missing data, it remains unclear whether these properties still hold. In particular, it is an open question whether Bayesian inference based on the (ignorable) likelihood remains valid when the data are not fully observed, even under the MAR assumption. 

\vspace{0.2cm}
Beyond parametric models, frequentist validation of Bayesian methods has been developed through the so-called \emph{prior mass and testing} framework. This approach provides general conditions ensuring posterior concentration in nonparametric models with quantifiable rates. The key ingredients are (a) sufficient prior mass in a KL neighborhood of the truth; (b) the existence of suitable tests; and (c) prior concentration on a sieve. This theory is now well understood in standard settings, but its extension to models with missing data remains largely unexplored. The question we consider here is thus whether such general nonparametric results can be extended to the case where only incomplete data are observed, and the true missingness mechanism is MAR. In particular, we are interested in the question whether the complete-data density can be recovered in this case.

\vspace{0.2cm}
A fundamental result asserts that under i.i.d.\ data and when $\Theta$ is a space of probability measures metrized by the Hellinger distance, the existence of suitable tests in (b) is guaranteed (see, e.g., \cite[Appendix B]{Fundamentals}). We show that, somewhat surprisingly, this remains true under MAR missingness and a positivity condition. We achieve this by constructing a specific test for any combination of densities, inspired by the ideas in \cite[Appendix B]{Fundamentals}. Moreover, since neither the prior nor the sieve is changed by missing values in our framework, (c) holds iff it holds in the complete data case. However, again somewhat surprisingly, the prior mass in the KL neighborhood might not hold with missing values, even if it holds in the complete case. As such, while (b) follows for the Hellinger distance, and (c) can be checked with complete data, (a) needs to be carefully investigated with missing data. We show that, nonetheless, (c) can be verified in the case of Bayesian regression and density estimation on $\R^d$. Making use of the connection between the KL divergence, likelihood and the MAR condition, our general results allow to derive a density estimator that is straightforward to implement, while attaining the minimax contraction rate.

Our contributions are thus fourfold:
\begin{itemize}
    \item[1.] We extend a general posterior contraction result to the case of non-monotone MAR missingness.
    \item[2.] We prove that even under non-monotone MAR missingness, suitable tests always exist for the Hellinger distance.
    \item[3.] We apply these results to show that 
    \begin{itemize}
        \item  density estimation with Dirichlet priors under a Hölder condition reaches a posterior contraction rate that corresponds to the minimax rate of estimation up to log factors,
        \item Bayesian nonparametric regression is consistent under standard assumptions, even if the covariates are masked by non-monotone MAR missingness.
    \end{itemize}

    \item[4.] We derive and implement an algorithm that takes incomplete data and produces a density estimate of the complete-data distribution. It is straightforward to sample from this density estimate and to calculate any parameter of interest. Through the use of the classical ignorability approach, this algorithm avoids the instability of inverse probability weighting under extreme propensity scores.
\end{itemize}
We note that our applications only scratch the surface as the results in 1. and 2. might be much more broadly applicable.

\section{Posterior concentration rates under MAR}\label{sec_contractionresults}


This section provides the main results of the paper. We first introduce a new variant of the Kullback-Leibler divergence adapted to the missing data setting and then state a general theorem formulated under the celebrated prior mass and testing framework. Finally, we specialize this result to the case of estimation in Hellinger distance.

\subsection{Definitions}


We first introduce the following adapted definition of the KL divergence:
\begin{dfn}[KL relative to $\mathbb{P}_{\theta^*}$]\label{definition_KL0}
We define for any $\theta \in \Theta$,
\[
\widetilde{\textnormal{KL}}(P_{\theta^*} \| P_{\theta}):=\E_{(X,M)\sim\mathbb{P}_{\theta^*}}\left[ \log \left( \frac{p_{\theta^*}^{(M)}(X^{(M)})}{p_{\theta}^{(M)}(X^{(M)})} \right) \right].
\]
$\widetilde{\textnormal{KL}}$ is referred to as the Kullback-Leibler (KL) divergence relative to $\mathbb{P}_{\theta^*}$.
\end{dfn}

Although never formally introduced before, this notion of KL divergence relative to $\mathbb{P}_{\theta^*}$ is quite natural as it is simply defined as the expected log-ratio of observed densities.
Note that the expectation in $\KLMAR$ is taken over $(M, X^{(M)})$, the incomplete data we actually observe, while we would like to consider the distance between the complete-data models $P_{\theta^*}$ and $P_{\theta}$. Despite this difficulty, $\KLMAR$ is nonnegative under MAR, which cannot be extended to MNAR missing mechanisms.

\begin{restatable}[KL relative to $\mathbb{P}_{\theta^*}$]{prop}{KLdef}\label{definition_KL}
Under Assumptions \ref{asm_true_MDM} and \ref{asm_no_empty_MDM}, the KL divergence relative to $\mathbb{P}_{\theta^*}$ satisfies
\[
\widetilde{\textnormal{KL}}(P_{\theta^*} \| P_{\theta}) \geq 0 \quad \text{for any } \theta \in \Theta, 
\]
and
\[
\widetilde{\textnormal{KL}}(P_{\theta^*} \| P_{\theta^*}) = 0 .
\]
In particular $\theta \mapsto \widetilde{\textnormal{KL}}(P_{\theta^*} \| P_{\theta})$ is minimized at $\theta^*$.  

\end{restatable}

The KL divergence relative to $\Pjoint$ thus mirrors the role of the KL divergence relative to the data generating process in complete-data settings with a misspecified model, where $\theta^*$ is not inherently defined but emerges as the minimizer of the relative KL. Remarkably, under MAR missing data, the true parameter $\theta^*$ naturally minimizes $\widetilde{\textnormal{KL}}(P_{\theta^*} \| P_{\theta})$ and becomes the unique minimizer if we also assume a positive probability of observing the full data:

\begin{restatable}[$\KLMAR$ Divergence Nature]{prop}{KLloss}\label{informationloss_KL}
Under Assumptions \ref{asm_true_MDM}, \ref{asm_no_empty_MDM}, for any $\theta \in \Theta$, it holds that
    \begin{align*}
     0 \leq   \KLMAR(P_{\theta^*}\| P_{\theta}) \leq  \KL(P_{\theta^*}\| P_{\theta}).
    \end{align*}
    Moreover, $\KL(P_{\theta^*}\| P_{\theta}) > 0$ implies $\KLMAR(P_{\theta^*}\| P_{\theta})>0$ under the additional Assumption \ref{asm_positive_MDM}.
    In this case, we thus have
\[
\widetilde{\textnormal{KL}}(P_{\theta^*} \| P_{\theta}) = 0 \quad \text{if and only if} \quad P_{\theta} = P_{\theta^*} .
\]
\end{restatable}

We note that Assumption \ref{asm_positive_MDM} need only be true for $\delta=0$. On the other hand, it can be shown that if Assumption \ref{asm_positive_MDM} does not hold for $\delta=0$, it is possible to construct a distribution $P_{\theta_1}$ such that $\KL(P_{\theta^*}\| P_{\theta_1}) > 0$, but $\KLMAR(P_{\theta^*}\| P_{\theta_1})=0$.

\vspace{0.2cm}
The above shows that $\KLMAR$ is nonnegative and zero for $\theta=\theta^*$. However, compared to the case of complete data, we lose information and the KL divergence is reduced. Following classical Bayesian theory as in \cite{Fundamentals}, we also define
\begin{align}
   \KLMARsq(P_{\theta^*}\| P_{\theta})= \mathbb{E}_{(X,M)\sim\mathbb{P}_{\theta^*}} \left[  \log \left(\frac{p_{\theta^*}^{(M)}\left(X^{(M)}\right)}{p_{\theta}^{(M)}\left(X^{(M)}\right)} \right)^2 \right] .
\end{align}
Unfortunately, contrary to $\KLMAR$, it is not straightforward to derive properties of $\KLMARsq$, beyond the fact that it is clearly nonnegative.



These definitions can then be used to define the $\KLMAR$ ball around $\theta^*$:


\[ \mathcal{B}(\theta^*,\varepsilon; \mathbb{P}_{\theta^*}) = \left\{ \theta \in \Theta : \KLMAR(P_{\theta^*}\| P_{\theta} )\leq \varepsilon^2 \, , \, \KLMARsq(P_{\theta^*}\| P_{\theta})\leq \varepsilon^2 \right\}, \]
for \(\varepsilon > 0\). This is a natural adaptation of the KL ball usually considered in the nonparametric Bayesian literature, see e.g., \cite[Chapter 8]{Fundamentals}. We now use these definitions to obtain general contraction results under MAR missingness.

\subsection{General Posterior Contraction Results}


We can now formulate our second group of assumptions:
\begin{asm}[Prior Mass and Sieve]
\label{asm_prior2plus}
There exist positive constants $\CI, \CII$, a sequence of measurable sets $(\Theta_n)_n \subset \Theta$, and sequences $\varepsilon_n >0$, $\bar{\varepsilon}_n > 0$, $\bar{\varepsilon}_n \leq \varepsilon_n$ with $n \bar{\varepsilon}_n \to \infty$ such that
\begin{enumerate}[label=\roman*), ref=\theasm~\roman*)]
    \item \label{v_priormassII} $\Pi\left(\mathcal{B}(\theta^*,\bar{\varepsilon}_n; \mathbb{P}_{\theta^*})\right) \geq e^{-\CI n \bar{\varepsilon}_n^2}$,
for small enough $\bar{\varepsilon}_n^2$.
  \item\label{iii_coveringnumber} $\log N(\varepsilon_n, \Theta_n, d) \leq \CII n \varepsilon_n^2$.
 \item \label{iv_priormassI} $\Pi(\Theta_n^c) \leq e^{-n \bar{\varepsilon}_n^2(\CI+4)}$.
\end{enumerate}
\end{asm}

Assumption \ref{asm_prior2plus} is closely related to the standard prior mass condition commonly used in the Bayesian nonparametrics literature. Part \ref{v_priormassII} requires that the prior assigns sufficient mass to a Kullback-Leibler neighborhood of the true parameter $\theta^*$. This is a natural requirement, since an insufficient prior mass in a neighborhood of $\theta^*$ would lead to a vanishing posterior mass. The only distinction here is that the neighborhood is defined using $\KLMAR$, whereas the classical formulation typically relies on the standard $\KL$ divergence.
Parts \ref{iii_coveringnumber} and \ref{iv_priormassI} correspond to the usual sieve conditions; see, for instance, Chapter 1 of \cite{castillo2025book}. These conditions assume the existence of a sequence of sieves $\Theta_n$ with increasing size and controlled complexity.

\vspace{0.2cm}
The following assumption introduces a testing condition ensuring the existence of suitable tests for identifying $\theta^*$.

\begin{asm}[Test]
\label{asm_test2}
There exist constants $\CIII > 0$ and $ a \in (0,1)$, such that for any $\varepsilon > 0$ and $\theta_1$ with $d(\theta^*, \theta_1) > \varepsilon$ there exists a sequence of tests $\phi_n=\phi_n\left(\data\right)$, such that
$$
\mathbb{E}_{\data\sim\mathbb{P}_{\theta^*}^n}\left[\phi_n(\data)\right] \leq  e^{-\CIII n \varepsilon^2} \quad , \quad \mathbb{E}_{\data\sim\mathbb{P}_{\theta}^n}\left[1-\phi_n(\data)\right] \leq e^{-\CIII n \varepsilon^2}, 
$$
for all $\theta$ such that $d(\theta, \theta_1) <  a \varepsilon$.
\end{asm}

This condition is analogous to the usual testing assumption in the Bayesian nonparametrics literature, with the notable difference that, in standard approaches, the test statistics are typically constructed using the complete dataset, whereas we only rely here on the incomplete observed data.
Under these assumptions, it is possible to establish a rate of convergence analogous to that of the complete-data setting: 

\begin{restatable}[General Rate Result]{thm}{Generalrates}\label{Thm:GeneralResult}
Assume that Assumptions \ref{asm_true_MDM}, \ref{asm_no_empty_MDM} and Assumptions \ref{asm_prior2plus}, \ref{asm_test2} hold on the semimetric space $(\Theta, d)$. Then, for $\CIV=\CIV(a, \CI, \CII, \CIII)$ large enough,
\begin{align}
\Pi\left(d(\theta , \theta^*) \geq \CIV \varepsilon_n \, \big| \, \data\right) \to 0 \quad \text{in } \mathbb{P}_{\theta^*} \text{-probability}.
\end{align}

\end{restatable}

Assumptions \ref{asm_true_MDM}--\ref{asm_no_empty_MDM} are assumed throughout this paper. To verify Assumptions \ref{iii_coveringnumber}-\ref{iv_priormassI}, standard theory can be used for well-studied spaces $\Theta_n$, see, e.g., \cite{Fundamentals} for several examples. 
The new challenge lies in the validation of Assumption \ref{asm_test2} and Assumption \ref{v_priormassII}, both involving the missing data part.

\subsection{Posterior Contraction Results in Hellinger distance}

In the complete data case and when the metric $d=d_H$ on $\Theta$ is induced by the Hellinger distance, it is well-known that the test condition Assumption \ref{asm_test2} is always met for i.i.d.\ data, see, e.g.\ Theorem 1.2 in \cite{castillo2025book}. Surprisingly, Theorem \ref{thm_testingconditionholds} shows that this also holds under MAR and positivity of the propensity score: 


\begin{thm}\label{thm_testingconditionholds}
    Assume Assumptions \ref{asm_true_MDM} -- \ref{asm_positive_MDM} hold. Moreover, assume that there exists $\theta_{\star} \in \Theta$ such that $x \mapsto p_{\theta_{\star}}(x)$ is continuous and $\supp(p_{\theta}) \subset \supp(p_{\theta_{\star}})$ for all $\theta \in \Theta$. Then Assumption \ref{asm_test2} holds on $(\Theta, d_H)$ with $\CIII=\delta/8$ and $a=\sqrt{\delta}/\sqrt{24}$.

\end{thm}

The proof is constructive, providing an explicit test satisfying the conditions of Theorem \ref{thm_testingconditionholds}. To construct this test, we first define an adapted version of the Hellinger distance $\He$ based on what we observe, similar to the $\KLMAR$ distance above. We refer to Appendix \ref{sec_proofs} for details. We also note that, similar to the case of the KL divergence in \Cref{informationloss_KL}, Assumption \ref{asm_positive_MDM} is crucial for testing. Indeed, if there exists a set $A$ with $P_{\theta^*}(A) > 0$ such that $\PMzerox=0$ for $x \in A$, then it is possible to construct a density $p_{\theta_1}$ that agrees with $p_{\theta^*}$ on $A^c$ and on all $k < d$ dimensional marginals, but for which $p_{\theta_1}(x)\neq p_{\theta^*}(x)$, for $x \in A$. Thus no test based on the observed data would be able to differentiate $P_{\theta^*}, P_{\theta_1}$. Assumption \ref{asm_positive_MDM} with $\delta > 0$ makes it possible to relate the observable difference between $P_{\theta^*}$, $P_{\theta}$ more precisely to the Hellinger distance between $p_{\theta^*}, p_{\theta}$, making it possible to construct tests of uniform power.

\vspace{0.2cm}
As a direct consequence of \Cref{Thm:GeneralResult}, and \Cref{thm_testingconditionholds}, we have:

\begin{prop}\label{Cor:GeneralResultH}
Assume that there exists $\theta_{\star} \in \Theta$ such that $x \mapsto p_{\theta_{\star}}(x)$ is continuous and $\supp(p_{\theta}) \subset \supp(p_{\theta_{\star}})$ for all $\theta \in \Theta$. Moreover, assume that Assumptions \ref{asm_true_MDM}--\ref{asm_positive_MDM} and Assumption \ref{asm_prior2plus} hold on $(\Theta, d_H)$.
Then, for $\CIV=\CIV(\CI, \CII)$ large enough: 
\begin{align*}
\Pi\left(\He(p_{\theta} , p_{\theta^*
}) \geq \CIV \frac{\varepsilon_n}{\sqrt{\delta}} \, \big| \, \data\right) \to 0 \quad \text{in } \mathbb{P}_{\theta^*} \text{-probability}.
\end{align*}
\end{prop}

This result can be interpreted as a natural extension of classical posterior concentration theorems to the setting with missing data. The parameter $\delta$, arising from the positivity assumption on the propensity score (Assumption~\ref{asm_positive_MDM}), quantifies the level of missingness in the data. In the complete-data setting, one has $\delta = 1$ (and $\KLMAR=\KL$), and the result reduces to the standard posterior contraction theorem in Bayesian nonparametrics (see, e.g., Theorem~1.3 in \cite{castillo2025book}). The proposition shows that posterior concentration in Hellinger distance still holds in the presence of missing data, with a convergence rate of order $\varepsilon_n / \sqrt{\delta}$. Importantly, the sequence $\varepsilon_n$ is of the same order as in the complete-data case, up to the fact that the prior mass condition is formulated in terms of the $\KLMAR$ divergence instead of the usual Kullback-Leibler divergence, which needs to be verified on a case-by-case basis. The additional factor $1/\sqrt{\delta}$ reflects the loss of information due to missingness, and is therefore unavoidable: as $\delta$ decreases, the effective sample size deteriorates, leading to slower convergence.

\vspace{0.2cm}
Finally, we note that if the prior mass condition Assumption \ref{v_priormassII} is met in the case of complete data, then, since $\KLMAR \leq \KL$,
\begin{align*}
   \{\theta \in \Theta:\KL(P_{\theta^*} \| P_{\theta}) \leq \varepsilon^2\}\subset \left \{\theta \in \Theta:  \KLMAR(P_{\theta^*} \| P_{\theta})\leq \varepsilon^2\right\}  .
\end{align*}
While this is likely enough to justify a Schwartz-type result similar to \cite[Theorem 6.23]{Fundamentals}, proving consistency without any rate, it is not sufficient for the stronger result we wish to prove here. The reason is that the same relation between complete and missing data quantities does not hold for $\KLMARsq$, and it is thus unclear whether meeting the prior mass assumption in the complete data case also implies that Assumption \ref{v_priormassII} is met under MAR missingness.
We now apply these general results in the case of density estimation and regression.

\section{Examples and Applications}\label{sec_examples}

In this section, we illustrate the consequences of our general posterior contraction theorem under MAR missingness through several classical Bayesian models. We first consider density estimation, which provides the main application of our theory. In this setting, we show that a Dirichlet mixture of normal prior achieves the same minimax contraction exponent as in the complete-data case, despite the fact that only incomplete observations are available. We then turn to regression models and distinguish two situations: missing responses and missing covariates. While missing responses lead to an almost immediate extension of complete-data results, missing covariates require a more careful analysis of the local geometry of the likelihood.

\vspace{0.2cm}

Throughout this section, $\mathcal X=\mathbb R^d$ and we assume that $d\geq2$, since the interest of general MAR mechanisms arises mainly in the multivariate setting. In particular, when $d=1$, missingness patterns are substantially simpler and do not allow for genuinely multivariate MAR structures.


\subsection{Density Estimation on $\R^d$}\label{sec_density}

The first application concerns the estimation of an unknown density from incomplete observations. This example is particularly relevant for our framework since the target of interest is the complete-data distribution, while the observed likelihood only involves partially observed realizations. We show that, despite this loss of information, the posterior contracts at the same minimax exponent as in the absence of missing values.

\vspace{0.2cm}

Let $\phi(\cdot, \Sigma)$ be the density of the $\Gauss{0}{\Sigma}-$ distribution where $\Sigma$ is a positive definite $d \times d$ matrix. Let moreover $I_d$ be the identity matrix in $d$ dimensions. We consider
\[
p_{\theta}(x)=\int \phi(x-z, \Sigma) d F(z).
\]
The prior distribution $\Pi$ on $\Theta$ is chosen as an inverse Wishart distribution for $\Sigma$ and, independently, a Dirichlet Process with parameter $\alpha$ and base measure $\Gauss{\mu}{I_d}$ for $F$, see, e.g., \cite[Chapter 4/5]{Fundamentals}, or \Cref{alg:dpgmm_shared} in Appendix \ref{App_Implementation} for details. We assume that the true distribution $p_{\theta^*}$ lies in a Hölder class:
\begin{asm}[Hölder Class]\label{asm_Hölder}
Let for $k=(k_1, \ldots, k_d) \in \mathbb{N}^d$, let  $|k|=\sum_{j=1}^{d} k_i$, and define the derivative operator
\[
D^k=\frac{\delta^{|k|} }{\delta x_1^{k_1} \cdots \delta x_d^{k_d}}.
\]
Then $p_{\theta^*}$ has mixed partial derivatives $D^k p_{\theta^*}$ of order up to $|k| \leq \underline{\beta} := \lceil \beta - 1 \rceil$, satisfying for a function $L: \mathbb{R}^d \to [0, \infty)$,
\begin{equation}
|D^k p_{\theta^*}(x_1 + x_2) - D^k p_{\theta^*}(x_1)| \leq L(x_1) e^{a\|x_1\|^2} \|x_2\|^{\beta - k}, \quad k = \underline{\beta}, \; x_1, x_2 \in \mathbb{R}^d,
\tag{9.9}
\end{equation}
\begin{equation}
\E_{X\sim P_{\theta^*}} \left[ \left(\frac{L(X)}{p_{\theta^*}(X)}\right)^2 + \left(\frac{|D^k p_{\theta^*}(X)|}{p_{\theta^*}(X)}\right)^{2\beta/k} \right] < \infty, \quad k \leq \underline{\beta}.
\tag{9.10}
\end{equation}
Furthermore, $p_{\theta^*}(x) \leq c e^{-b\|x\|^{\tau}}$, for every $\|x\| > a$, for positive constants $a, b, c, \tau$.
\end{asm}

The following result shows that MAR missingness does not modify the minimax contraction exponent of the posterior distribution achievable in the complete-data setting under the condition above:

\begin{restatable}[Density Estimation Rate]{thm}{Densityrates}\label{Thm:DensityResult}
    Assume that $\Pi$ is Dirichlet mixture of normal prior described above, that $d \geq 2$ and that Assumptions \ref{asm_true_MDM} -- \ref{asm_positive_MDM} and Assumption \ref{asm_Hölder}, with parameters $\beta, \tau$, hold. Then for $\CIV$ large enough:
\begin{align}
\Pi\left(\He(p_{\theta} , p_{\theta^*
}) \geq \frac{\CIV}{\sqrt{\delta}} n^{-\beta/(2\beta + d)} (\log(n))^t \, \big| \, \data\right) \to 0 \quad \text{in } \mathbb{P}_{\theta^*} \text{-probability},
\end{align}
with 
\[
t > \frac{\beta d + \beta d/\tau + d + \beta}{2 \beta + d}.
\]
\end{restatable}

The rate achieved is the same as that presented in \cite[Theorem 9.9]{Fundamentals}, at the price of the factor $1/\sqrt{\delta}$. In particular, the minimax rate is attained up to log-factors, see, e.g., \cite[Chapter 9.4]{Fundamentals}. The result illustrates the main message of this paper: under MAR missingness and a mild positivity assumption, the unknown distribution of the complete data remains consistently estimable at the optimal nonparametric rate.

\subsection{Regression with Missing Responses} 

We now consider Bayesian regression, where the covariates are always observed while the response variable is subject to MAR missingness. We assume that the distribution of the covariates is fixed and known and only model the conditional distribution of the response given the covariates. 

\begin{asm}[Missing Response]\label{asm_Missing_response}
Let $(X_i,Y_i)$, $i=1,\ldots,n$, be i.i.d. observations with $X_i\in\mathbb{R}^d$ and $Y_i\in\mathbb R$. We assume that the covariates are fully observed, while $Y_i$ may be missing according to a MAR mechanism satisfying Assumption \ref{asm_true_MDM}. We further assume that the marginal distribution of the covariates $X$ is fixed and known while only the conditional distribution of the response depends on the parameter $\theta$. We finally assume that all distributions in the model admit a continuous positive density with respect to some Lebesgue's measure.
\end{asm}

Contrary to the density estimation setting, posterior contraction results established for complete-data regression models can be transferred to the MAR setting with essentially no additional work. We note that this setting of missing response with fully observed covariates is the one most often considered in the literature. 

\begin{restatable}[Regression with Missing Response]{thm}{MissingResponseRates}\label{Thm:MissingResponseRate}
    Suppose that Assumptions \ref{asm_true_MDM} -- \ref{asm_positive_MDM} and Assumption \ref{asm_Missing_response} hold. Assume moreover that the prior satisfies the classical prior mass condition for complete-data, that is that there exist positive constants $\CI, \CII$, a sequence of measurable sets $(\Theta_n)_n \subset \Theta$, and sequences $\varepsilon_n >0$, $\bar{\varepsilon}_n > 0$, $\bar{\varepsilon}_n \leq \varepsilon_n$ with $n \bar{\varepsilon}_n \to \infty$ such that
\begin{enumerate}[label=\roman*), ref=\theasm~\roman*)]
    \item $\Pi\left(\mathcal{B}(\theta^*,\bar{\varepsilon}_n)\right) \geq e^{-\CI n \bar{\varepsilon}_n^2}$,
for small enough $\bar{\varepsilon}_n^2$, 
 \item\label{iii_coveringnumber_complete} $\log N(\varepsilon_n, \Theta_n, d_H) \leq \CII n \varepsilon_n^2$,
 \item  $\Pi(\Theta_n^c) \leq e^{-n \bar{\varepsilon}_n^2(\CI+4)}$, 
\end{enumerate}
where
\[
\mathcal{B}(\theta^*,\varepsilon) = \left\{ \theta \in \Theta : \KL(P_{\theta^*} \| P_{\theta} )\leq \varepsilon^2 \, , \, \mathbb{E}_{P_{\theta^*}} \left[ \log \left( \frac{\mathrm{d}P_{\theta^*}}{\mathrm{d}P_{\theta}} \right)^2 \right]\leq \varepsilon^2 \right\} \, .
\]
Then,
\[
\Pi\!\left( \He(P_\theta,P_{\theta^*}) \ge C_4\frac{\varepsilon_n}{\sqrt{\delta}} \,\middle|\,\mathcal{S}_n \right) \longrightarrow 0 
\]
in $P_{\theta^*}$-probability. 
\end{restatable} 

This result shows that MAR missing responses do not introduce any additional difficulty beyond the verification of the standard complete-data prior conditions. Indeed, the observed-data likelihood ratio coincides with the complete-data likelihood ratio on observed responses and does not contribute to additional information loss when the response is missing.

\vspace{0.2cm}

Therefore, any posterior contraction theorem available for Bayesian regression under complete observations immediately extends to the MAR response-missing setting through our general result.

\subsection{Gaussian Regression with Missing Covariates}

We finally consider the more challenging regression setting where some components of the covariates are missing while the response remains fully observed. 

\begin{asm}[Missing Covariates]\label{asm_Missing_covariates}
Let $(X_i,Y_i)$, $i=1,\ldots,n$, be i.i.d. observations such that for some known noise variance $\sigma^2>0$:
\[
Y_i=f_{\theta^\star}(X_i)+\varepsilon_i,\qquad
\varepsilon_i\stackrel{i.i.d.}{\sim}\mathcal{N}(0,\sigma^2) \, .
\]
We assume that the response is always observed, while some features of the covariates $X_i$ may be missing according to a MAR mechanism satisfying Assumption \ref{asm_true_MDM}. We further assume that the marginal distribution of the covariates $X$ is fixed and known while only the conditional distribution of the response is modeled through the equation $Y=f_\theta(X)+\varepsilon$, with $\varepsilon\sim\mathcal{N}(0,\sigma^2)$, $f_\theta(\cdot)$ continuous, and $\|f_\theta\|_\infty\leq B$ for some $B>0$ and any $\theta$. 
\end{asm}

Contrary to the previous case, marginalizing over the missing covariates modifies the likelihood itself, and the connection with the complete-data model is no longer immediate. Nevertheless, under the assumption that the distribution of the covariates is known, we show that complete-data prior concentration conditions can still be transferred to the MAR setting. The key point is that a local neighborhood in the supremum norm for the regression function induces a corresponding neighborhood for the observed-data distributions.

\begin{restatable}[Regression with Missing Covariates]{thm}{MissingCovariatesRates}\label{Thm:MissingCovariatesRate}
    Suppose that Assumptions \ref{asm_true_MDM} -- \ref{asm_positive_MDM} and Assumption \ref{asm_Missing_covariates} hold. Assume moreover that the prior satisfies the following variant of the prior mass condition for complete-data: there exist positive constants $c, \CI, \CII$, a sequence of measurable sets $(\Theta_n)_n \subset \Theta$, and sequences $\varepsilon_n >0$, $\bar{\varepsilon}_n > 0$, $\bar{\varepsilon}_n \leq \varepsilon_n$ with $n \bar{\varepsilon}_n \to \infty$ such that
\begin{enumerate}[label=\roman*), ref=\theasm~\roman*)]
    \item \label{v_priormassII_complete} $\Pi\left(\theta \in \Theta : \|f_\theta-f_{\theta^*} \|_\infty \leq c\bar{\varepsilon}_n^2\right) \geq e^{-\CI n \bar{\varepsilon}_n^2}$,
for small enough $\bar{\varepsilon}_n^2$,
  \item\label{iii_coveringnumber_complete} $\log N(\varepsilon_n, \Theta_n, d_H) \leq \CII n \varepsilon_n^2$,
 \item \label{iv_priormassI_complete} $\Pi(\Theta_n^c) \leq e^{-n \bar{\varepsilon}_n^2(\CI+4)}$ \, .
\end{enumerate}
Then,
\[
\Pi\!\left( \He(P_\theta,P_{\theta^*}) \ge C_4\frac{\varepsilon_n}{\sqrt{\delta}} \,\middle|\,\mathcal{S}_n \right) \longrightarrow 0 
\]
in $P_{\theta^*}$-probability. 
\end{restatable} 

We note that the Hellinger distance between $P_\theta$ and $P_{\theta^*}$ is related to the discrepancy between the functions $f_\theta$ and $f_{\theta^*}$ themselves through the formula
$$
\He^2(p_\theta,p_{\theta^*}) = \E_{X}\left[1-\exp\left(\frac{(f_\theta(X)-f_{\theta*}(X))^2}{8\sigma^2}\right)\right] \, .
$$

The proof of Theorem \ref{Thm:MissingCovariatesRate} is deferred to Appendix \ref{sec_proofs}. The main ingredient is to show that a prior concentration condition formulated in the supremum norm for the complete-data regression function implies the corresponding local prior mass condition required by our general MAR contraction theorem.

\vspace{0.2cm}


Supremum-norm concentration conditions are standard in Bayesian nonparametric regression, in particular for Gaussian process priors and sparse Bayesian neural networks; see e.g. \cite{castillo2025deep,castillo2025book,cherief2020convergence,giordano2023besov,ohn2024adaptive,polson2018posterior}. As such, we can apply Theorem \ref{Thm:MissingCovariatesRate} directly in these cases. One subtle difference is that most existing results for complete data formulate these assumptions in terms of squared supremum norms, whereas our argument requires concentration directly in the supremum norm. Hence, the prior condition appearing in Theorem \ref{Thm:MissingCovariatesRate} is slightly stronger than the usual complete-data assumptions.

\section{On Estimation and Imputation}\label{Sec_Motivating}\label{sec_empirical}


For missing data, density estimation plays a crucial role. With the procedure in \Cref{alg:dpgmm_shared}, we naturally arrived at a method that is able to make use of MAR missing data and produce new samples from the (estimated) complete-data distribution. This sample can then be used to estimate a wide variety of estimators of interest in a second step. Imputation methods (\cite{mice-rf-method,OTimputation, ImputationScores, näf2024goodimputationmarmissingness, MIRI}), generative methods (\cite{misgan, missdiff}), and the kernel density estimation approach in \cite{nonparametricpmm}, all try implicitly or explicitly to obtain such a data sample with distribution $P_{\theta^*}$ from incomplete observations. 


In this section, our goal is to illustrate this principle and empirically demonstrate that the performance of the theoretical algorithm implemented in \Cref{alg:dpgmm_shared} with missing values is roughly the same as the same algorithm using \emph{the complete data}. This is not a fair comparison, of course, as the complete data algorithm has access to more data than \Cref{alg:dpgmm_shared}. However, our use of a relatively high sample size $(n \in \{500, 1000\})$ compared to the dimension of $d=3$, should reveal a similar performance of the two.

\vspace{0.2cm}
We consider a simple example, inspired by \cite{näf2024goodimputationmarmissingness, practical}, with varying joint distribution $P_{\theta^*}$ of $(X_1, X_2, X_3)$, and the following missingness mechanism: For $\{m_1,m_2, m_3\}=\{(0,0,0), (0,1,0), (1,0,0)\}$ and $1 > \delta > 0$ small,
        \begin{align}\label{eq_MARmissing0}
        &\Prob(M=m_1\mid X=x)=(G(x_1)+G(x_2))/3, \nonumber \\
        &\Prob(M=m_2\mid X=x)=(2-G(x_1))/3 \nonumber \\
        &\Prob(M=m_3\mid X=x)=(1-G(x_2))/3,
    \end{align}
    where we choose
    \begin{align}\label{eq_Gx}
            G(x)=\max(\Phi(x), 2\delta),
    \end{align}
    with $\Phi$ is the standard Gaussian cdf. This mechanism clearly meets Assumption \ref{asm_true_MDM}. Moreover, it meets the positivity assumption (Assumption \ref{asm_positive_MDM}), in that $\Prob(M=m_1\mid X=x)> \delta > 0$. Nonetheless, as outlined in \cite[Example 6]{näf2024goodimputationmarmissingness}, this is a rather complex non-monotone mechanism-- simply ignoring the missing data will lead to a bias and it is unclear how to estimate $\Prob(M=m_1\mid X=x)$ for IPW strategies. 
 We attempt to (1) estimate the 0.1-quantile of $X_1$, which was the original ``challenge'' formulated in \cite{näf2024goodimputationmarmissingness, practical} and (2) to create a sample that is close in distribution to the (unobserved) complete-data distribution. Since these are simulated examples, (2) can be measured by comparing the generated distribution with a new sample from the data generating process, using any distributional metric. In particular, we use the energy distance \citep{EnergyDistance}. We note that even if the goal is only (1), competitors for our method are difficult to find due to the challenging MAR mechanism. One exception is a subset of nonparametric multiple imputation by chained equation (mice, \cite{mice}) algorithms that, while not theoretically grounded, have shown superior performance in practice (see e.g. \cite{OneBenchmarktorulethemall}), among which we include ``mice$\_$rf'' (\cite{mice-rf-method}) as a representative choice. In fact, we modeled \Cref{alg:dpgmm_shared} in Appendix \ref{App_Implementation} closely to the theoretical analysis, and as such, we do not expect to beat this performance using our algorithm on real data. Further studies might uncover more competitive versions of \Cref{alg:dpgmm_shared}.

\vspace{0.2cm}
We start with $P_{\theta^*}=\Gauss{0}{\Sigma}$ with
\begin{align*}
    \Sigma=\begin{pmatrix}
1 & 0.7 & 0 \\
0.7 & 1 & 0 \\
0 & 0 & 1
\end{pmatrix}.
\end{align*}
Next, we also study a mixture of Gaussians $P_{\theta^*}=w_1\Gauss{\mu_1}{\Sigma}+ w_2 \Gauss{\mu_2}{\Sigma} + w_3 \Gauss{\mu_3}{\Sigma} $, with weights $w_1=0.3, w_2=0.4, w_3=0.3$, means $\mu_1=\begin{pmatrix} -3& 0& 0 \end{pmatrix}$, $\mu_2=\begin{pmatrix} 0 & 3& 0 \end{pmatrix}$, $\mu_3=\begin{pmatrix} -3& 0& 0 \end{pmatrix}$ and $\Sigma$ as above. Finally, we also study the original example formulated in \cite{näf2024goodimputationmarmissingness}, where $\Pjoint$ has uniform marginals, with a copula-induced dependence between $X_1$ / $X_2$ and $G(x)=x$. We note that this last example does not meet Assumption \ref{asm_Hölder}, nor does it meet \Cref{asm_positive_MDM}, though $\PMzerox > 0$ for all $x \in \mathcal{X}$ is still true.

Figures \ref{fig:Gaussexample} - \ref{fig:uniformexample} show the results. As we use a Dirchlet mixture of normal, the first two settings are ideal for our method, though we note that this is still different than simply using a parametric approach. As such, our algorithm performs very well, both in estimating the quantile (1) and in terms of the energy distance (2). In particular, its performance is comparable to that of the algorithm having access to complete data, despite the tricky MAR mechanism and the loss of information. As expected, mice$\_$rf is also highly competitive, although it is somewhat struggling in terms of quantile estimation. On the other hand, for the uniform example the Dirchlet mixture of normal naturally does not perform as well, as the method tries to approximate a non-smooth density with smooth Gaussian distributions. Nonetheless, the performance for $n=1'000$ is again quite comparable, for our proposed algorithm with missing values and the same algorithm having access to the complete data.

\begin{figure}
    \centering
    \begin{subfigure}{0.49\textwidth}
        \centering
        \includegraphics[width=0.7\textwidth]{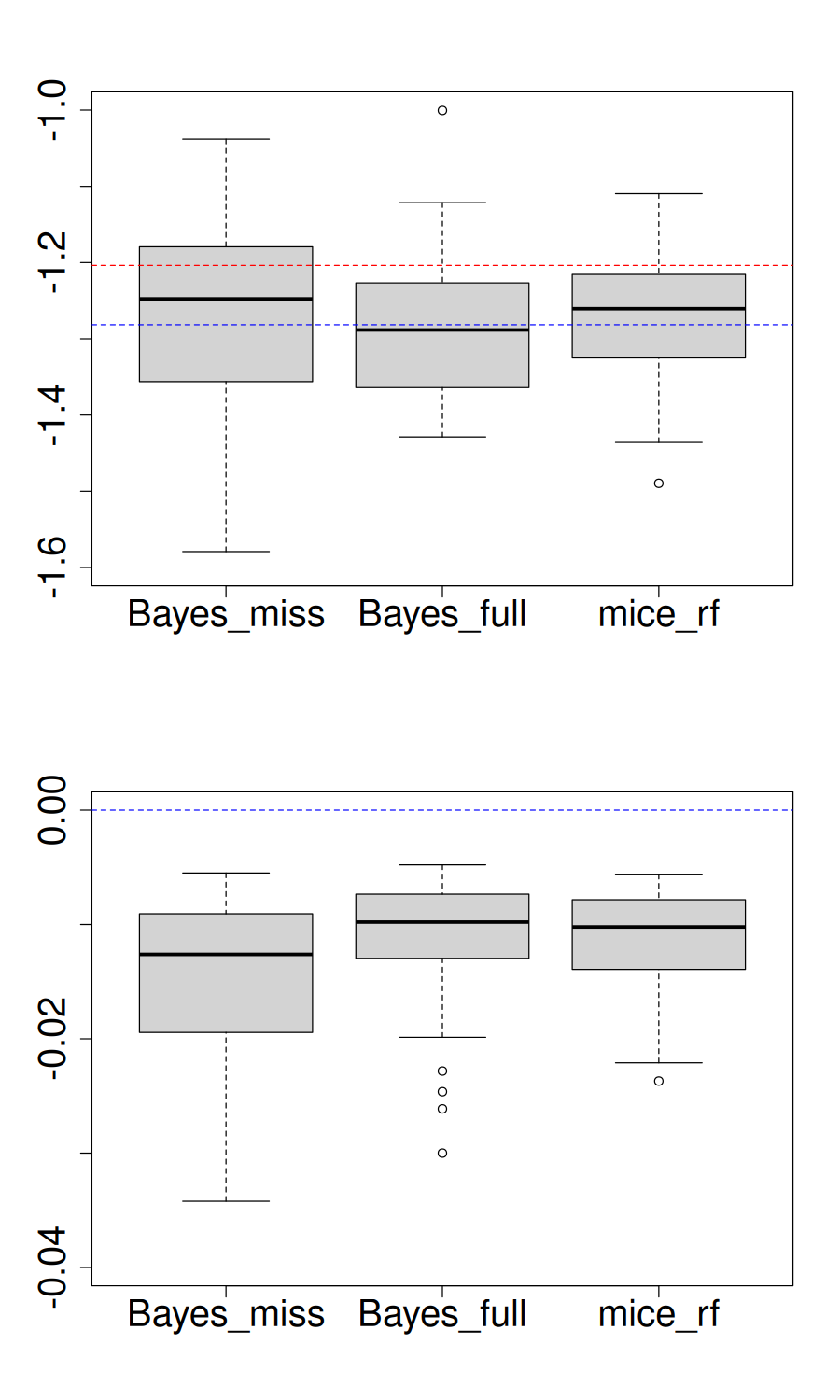}
    \end{subfigure}
    \hfill
    \begin{subfigure}{0.49\textwidth}
        \centering
        \includegraphics[width=0.7\textwidth]{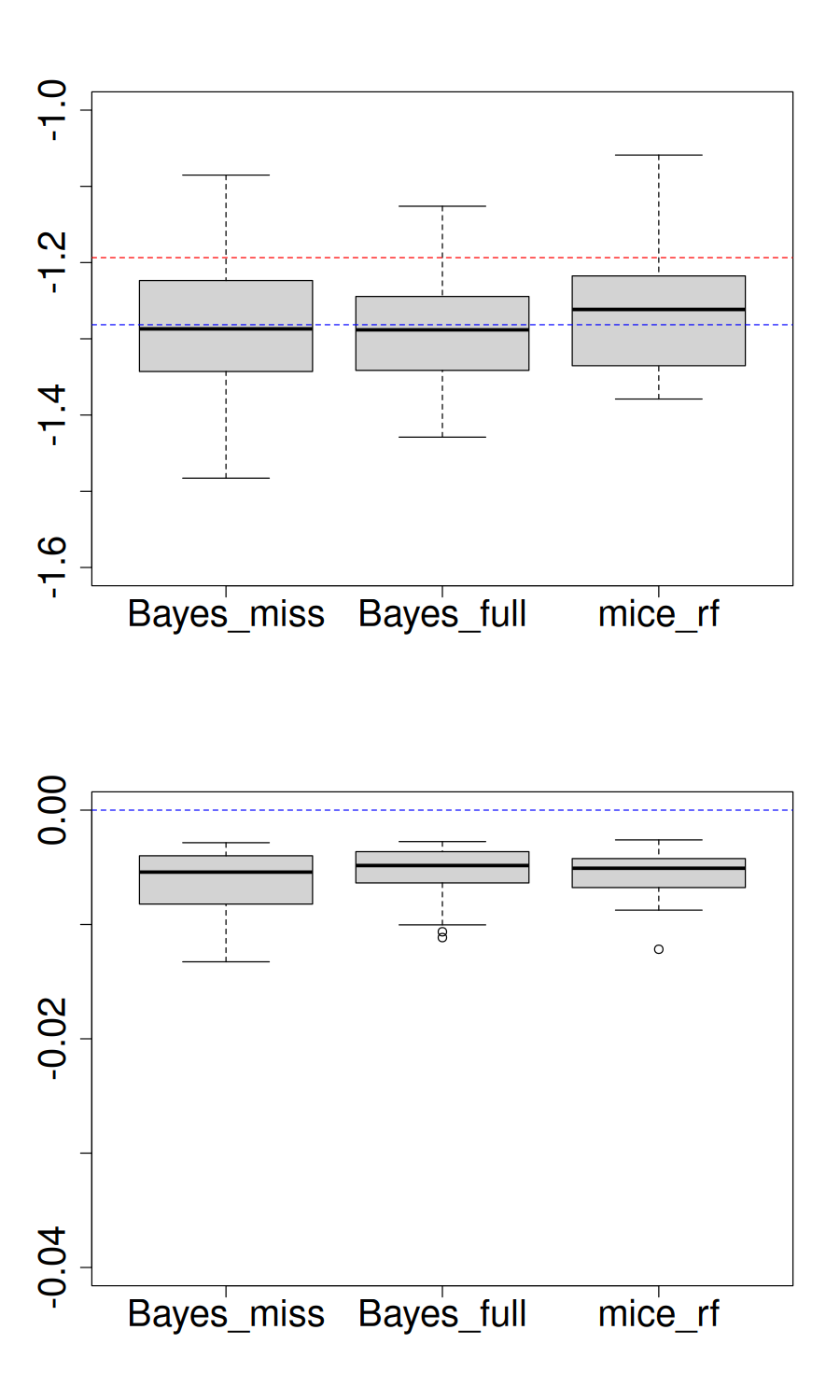}
    \end{subfigure}
    \caption{Example with $P_{\theta^*}$ chosen to be $\Gauss{0}{\Sigma}$. Top: Quantile estimate of $X_1$ for $n=500$ (left) and $n=1000$ (right), Bottom: negative energy distance between the newly generated sample/imputation and the complete data for $n=500$ (left) and $n=1000$ (right).}
    \label{fig:Gaussexample}
\end{figure}

\begin{figure}
    \centering
    \begin{subfigure}{0.49\textwidth}
        \centering
        \includegraphics[width=0.7\textwidth]{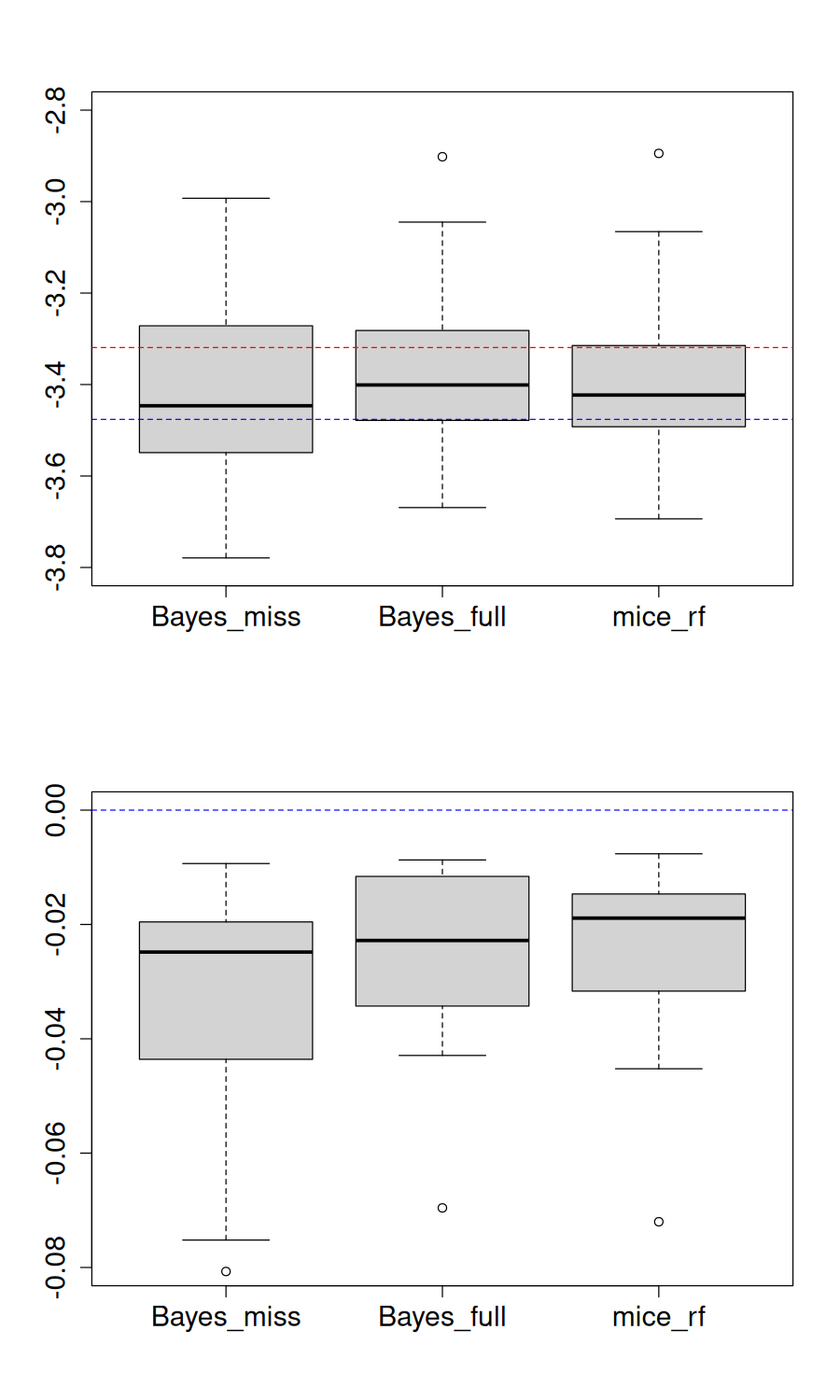}
    \end{subfigure}
    \hfill
    \begin{subfigure}{0.49\textwidth}
        \centering
        \includegraphics[width=0.7\textwidth]{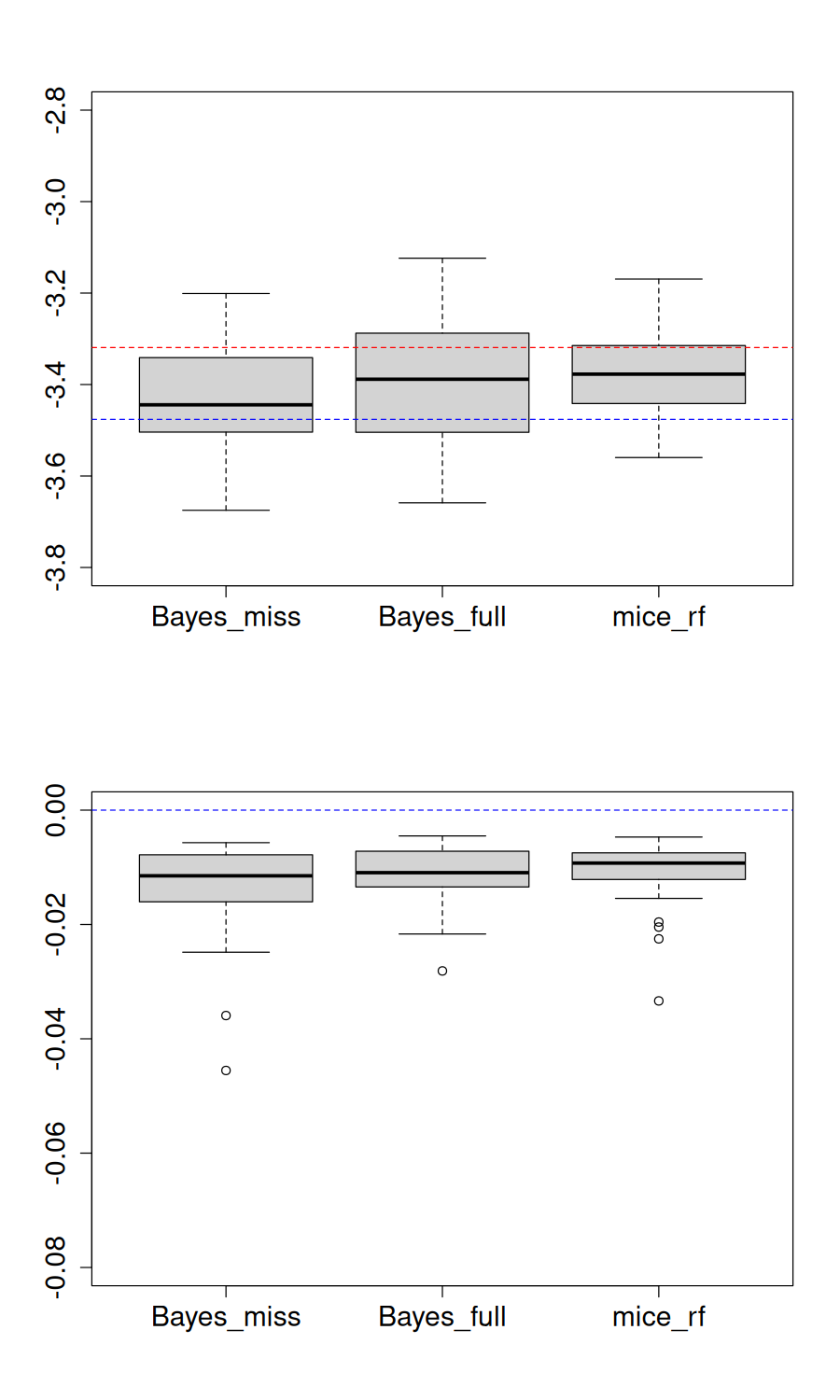}
    \end{subfigure}
    \caption{Example with $P_{\theta^*}$ chosen to be a mixture of Gaussians with different means and a correlation of $0.7$ between $X_1$ and $X_2$. Top: Quantile estimate of $X_1$ for $n=500$ (left) and $n=1000$ (right), Bottom: negative energy distance between the newly generated sample/imputation and the complete data for $n=500$ (left) and $n=1000$ (right).}
    \label{fig:GaussMixtureexample}
\end{figure}

\begin{figure}
    \centering
    \includegraphics[width=0.7\linewidth]{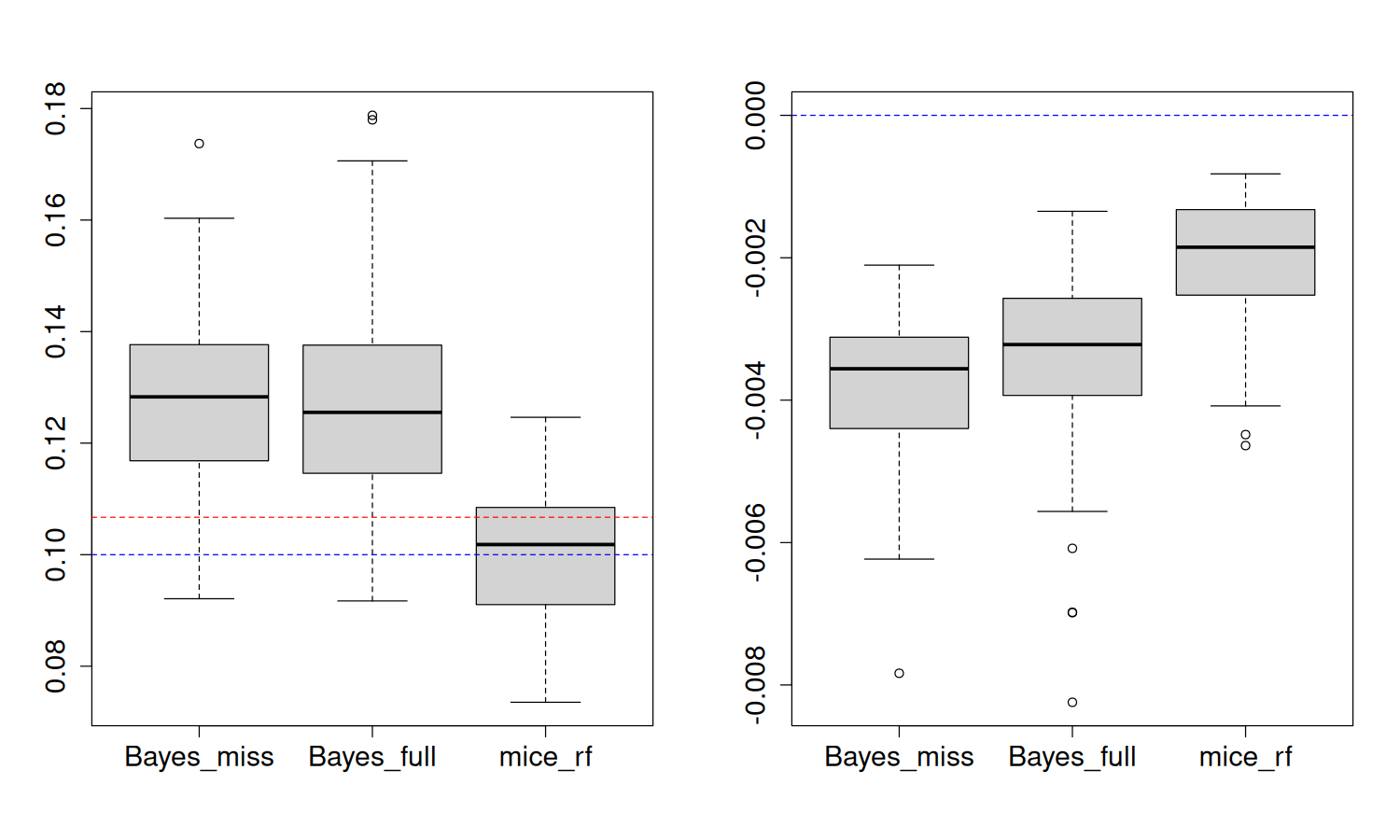}
    \caption{Example with $P_{\theta^*}$ chosen to be uniform on $[0,1]^3$ with correlation between $X_1$ and $X_2$ induced by a copula. Left: Quantile estimate of $X_1$, Right: negative energy distance between the newly generated sample/imputation and the complete data. We used $n=1000$.}
    \label{fig:uniformexample}
\end{figure}

\section{Conclusion}\label{sec_conclusion}

In this work, we adapted Bayesian posterior contraction results to the case of general non-monotone MAR missingness. We showed that the Hellinger distance still automatically meets the testing condition when a positivity assumption is added to MAR, providing a version of the result in which the testing condition could be removed. Applications of this theory to density estimation and regression lead to the seemingly first nonparametric consistency results under this general MAR condition. 

\vspace{0.2cm}
We believe that this work has only scratched the surface. In particular, Theorem \ref{Thm:GeneralResult} might be much more widely applicable. Moreover, a natural immediate question is whether a (semiparametric) Bernstein-von-Mises result can be derived. Such a result would allow to asymptotically approximate posterior distributions of parameters of interest by multivariate Gaussians and thus allow for principled uncertainty quantification under MAR missing values. We intend to study these questions in subsequent work. Finally, while the proposed density estimation algorithm works quite well in our simulations, it is somewhat slow. A simpler covariance structure for $\Sigma$, such as constraining it to be diagonal, is likely to substantially increase the speed of calculations.

\clearpage

\appendix

\section{Density Estimation with Missing Values}\label{App_Implementation}

Here we present the algorithm used to estimate the density, adapted from \cite{DPAlgorithm} and \cite[Chapter 5]{Fundamentals}.\footnote{A first template of the code and the algorithmic environments below, without missing values, was obtained using Claude AI.} We heavily rely on Metropolis Hastings algorithms, as described e.g. in \cite[Chapter 24.3]{Murphy}. The \texttt{R} implementation of the algorithm can be found on \url{https://github.com/JeffNaef/Bayesian-MAR}.

\begin{algorithm}
\caption{DP Gaussian Mixture with Shared Covariance}
\begin{algorithmic}[1]\label{alg:dpgmm_shared}
\Require Data $\data = \{X_1^{(M_1)}, \ldots, X_n^{(M_n)}\} \subset \mathbb{R}^d$
\Require Hyperparameters: $\alpha > 0$, $\mu_0 \in \mathbb{R}^d$, $\tau_0 > 0$, $\nu_0 > d-1$, $\Psi_0 \in \mathbb{R}^{d \times d}$
\Require Iterations: $T_{\text{total}}$, burn-in: $T_{\text{burn}}$
\Ensure Posterior samples

\State \textbf{Initialize:}
\State $z \gets$ sample(1:5, $n$), $K \gets \max(z)$
\State Sample $\mu_k \sim \mathcal{N}(\mu_0, \tau_0 I)$ for $k = 1, \ldots, K$
\State Sample $\Sigma \sim \mathcal{W}^{-1}(\Psi_0, \nu_0)$ \Comment{Single shared $\Sigma$}

\For{$t = 1$ to $T_{\text{total}}$}
    \State \Call{UpdateAssignmentsShared}{$\data, z, K, \{\mu_k\}_k, \Sigma, \alpha, \mu_0, \tau_0$} \Comment{Algorithm \ref{alg:update_assignments_shared}}
    \State Remove empty clusters; relabel; update $K$
    \State \Call{UpdateCenters}{$\data, z, K, \mu_0, \tau_0, \Sigma, \{\mu_k\}_k$} $\rightarrow \{\mu_k\}_{k=1}^K$ \Comment{Algorithm \ref{alg:update_centers}}
    \State \Call{UpdateSharedSigma}{$\data, z, K, \{\mu_k\}_k, \Psi_0, \nu_0, \Sigma$} $\rightarrow \Sigma$ \Comment{Algorithm \ref{alg:update_shared_sigma}}
    
    \If{$t > T_{\text{burn}}$}
        \State Save $z^{(t)}, K^{(t)}, \{\mu_k^{(t)}\}_{k=1}^{K^{(t)}}, \Sigma^{(t)}$
    \EndIf
\EndFor

\State \Return Posterior samples
\end{algorithmic}
\end{algorithm}


\begin{algorithm}
\caption{UpdateAssignmentsShared$(\data, z, K, \{\mu_k\}_k, \Sigma, \alpha, \mu_0, \tau_0)$}
\label{alg:update_assignments_shared}
\begin{algorithmic}[1]

\For{$i = 1$ to $n$}
    \State $z_i \gets \texttt{NA}$ \Comment{Remove observation $i$}
    \State $n_k \gets |\{j : z_j = k, j \neq i\}|$ for $k = 1, \ldots, K$
    
    \State \textbf{Compute log-probabilities:}
    \For{$k = 1$ to $K$}
        \If{$n_k > 0$}
            \State $\ell_k \gets \log n_k + \log \phi\left(X_i^{(M_i)}-\mu_k^{(M_i)}, \Sigma^{(M_i)}\right) $  \Comment{CRP + likelihood} 
        \Else
            \State $\ell_k \gets -\infty$
        \EndIf
    \EndFor
    
    \State \textbf{New cluster probability:}
    \State Sample $\mu_{\text{new}} \sim \mathcal{N}(\mu_0, \Sigma + \tau_0^2 I_d)$ 
    \State $\ell_{K+1} \gets \log \alpha + \log \phi\left( X_i^{(M_i)}-\mu_{\text{new}}^{(M_i)},  \Sigma^{(M_i)}\right) $ 
    
    \State $p \gets \text{softmax}(\ell_1, \ldots, \ell_{K+1})$
    \State Sample $z_i \sim \text{Categorical}(p)$
    
    \If{$z_i = K + 1$} \Comment{Create new cluster}
        \State $K \gets K + 1$
        \State $\mu_K \gets \mu_{\text{new}}$
    \EndIf
\EndFor
\end{algorithmic}
\end{algorithm}

\begin{algorithm}
\caption{UpdateCenters$(\data, z, K, \mu_0,\tau_0, \Sigma, \{\mu_k\}_k)$}
\label{alg:update_centers}
\begin{algorithmic}[1]
\Require Shared covariance $\Sigma$

\For{$k = 1$ to $K$}
    \State $\mathcal{I}_k \gets \{i : z_i = k\}$, $n_k \gets |\mathcal{I}_k|$
    
    \If{$n_k > 10$}
        
        \State \textbf{Sample from posterior:}
        \State $\mu_k \gets$ MHmu($\data$,$\mu_0$, $\tau_0$, $\Sigma$, $\mu_k$)
    \EndIf
\EndFor

\State \Return $\{\mu_k\}_{k=1}^K$
\end{algorithmic}
\end{algorithm}


\begin{algorithm}
\caption{Jointlik$(X,z, \{\mu_k\}_k,\Sigma)$}
\label{alg:Jointlik}
\begin{algorithmic}[1]

\State \textbf{Compute total scatter across ALL clusters:}
\State $K \gets \max(z)$ 
\For{$k = 1$ to $K$}
    \State $\mathcal{I}_k \gets \{i : z_i = k\}$, $n_k \gets |\mathcal{I}_k|$

        \State \textcolor{gray}{// Scatter around cluster mean}
        \State  $\ell_k \gets\sum_{i \in \mathcal{I}_k}\log \phi\left(X_i^{(M_i)}-\mu_k^{(M_i)}, \Sigma^{(M_i)}\right)$  
\EndFor

\State \Return $\sum_{k=1}^{K} \ell_k$
\end{algorithmic}
\end{algorithm}

\begin{algorithm}
\caption{UpdateSharedSigma$(\data, z, K, \{\mu_k\}_k, \Psi_0, \nu_0, \Sigma_{init})$}
\label{alg:update_shared_sigma}
\begin{algorithmic}[1]

\State \textbf{Initialize:} Set $\theta_{init} = \text{C}(\Sigma_{init})$, with C as in \eqref{C_func}
\State Evaluate $\ell^{(0)} = \log \mathcal{W}^{-1}(\Sigma_{init} \mid \Psi_0, \nu_0) + \text{Jointlik}(\data, z, \{\mu_k\}_k,\Sigma_{init}) + \log |J(\theta_{init})|$
\For{$t = 1, \ldots, T$}
    \State \textbf{Propose:} $\theta^* \sim \mathcal{N}(\theta^{(t-1)}, \sigma^2 I_{d(d+1)/2})$
    \State \textbf{Transform:} $\Sigma^* = C^{-1}(\theta^*)$
    \State \textbf{Evaluate target:}
    \Statex \hspace{2em} $\ell^* = \log \mathcal{W}^{-1}(\Sigma^* \mid \Psi_0, \nu_0) + \text{Jointlik}(\data, z, \{\mu_k\}_k,\Sigma^*) + \log |J(\theta^*)|$
    \State \textbf{Compute acceptance ratio:} $\alpha = \min(1, \exp(\ell^* - \ell^{(t-1)}))$
    \State \textbf{Accept/reject:} Draw $u \sim \text{Uniform}(0, 1)$
    \If{$u \leq \alpha$}
        \State $\theta^{(t)} \gets \theta^*$, $\Sigma^{(t)} \gets \Sigma^*$, $\ell^{(t)} \gets \ell^*$
    \Else
        \State $\theta^{(t)} \gets \theta^{(t-1)}$, $\Sigma^{(t)} \gets \Sigma^{(t-1)}$, $\ell^{(t)} \gets \ell^{(t-1)}$
    \EndIf
\EndFor
\State \Return $\Sigma^{(T)}$ 
\end{algorithmic}
\end{algorithm}

For the proposal of $\Sigma$ in Algorithm \ref{alg:update_shared_sigma}, we use the log-Cholesky parametrization, as in \cite{Pinheiro1996}: For a given candidate matrix $\Sigma$, we first use the Cholesky Decomposition $\Sigma = LL^\top$ where $L$ is the \emph{unique} lower Cholesky factor with positive diagonal elements \citep{Pinheiro1996}. Then we define the unconstrained vector $\theta \in \mathbb{R}^{d(d+1)/2}$ by:
\[
\theta_{ij} = \begin{cases}
\log L_{ii} & \text{if } i = j \\
L_{ij} & \text{if } i > j
\end{cases}
\]
Thus if we define $L=Chol(\Sigma)$ it holds that 
\begin{align}\label{C_func}
    C: \R^{d \times d} \to \R^{d(d+1)/2}, \ \  C(\Sigma)_{ij} = \begin{cases}
\log (Chol(\Sigma)_{ii}) & \text{if } i = j \\
Chol(\Sigma)_{ij} & \text{if } i > j
\end{cases}
\end{align}
It is then also straightforward to define the inverse transformation $C^{-1}: \R^{d(d+1)/2} \to \R^{d \times d} $. Finally, since we are now moving on $\theta$ instead of $\Sigma$, we also need to derive the Jacobian correction. It can be shown to be: 
\[
\log |det J(\theta)| = \sum_{i=1}^d (d - i + 2)\, \theta_{ii}
\]
where $\theta_{ii} = \log L_{ii}$ are the log-diagonal entries. The proposal distribution of the MH sampler is then simply 
\[
\theta^* \sim \mathcal{N}(\theta^{(t-1)}, \sigma^2 I_{d(d+1)/2}),
\]
where $I_{d}$ is the identity matrix of dimension $d$.

\vspace{1em}


\begin{algorithm}[H]
\caption{MHmu($\data$,$\mu_0$, $\tau_0$, $\Sigma$, $\mu_{init}$)}
\label{alg:mh_mu}
\begin{algorithmic}[1]

\State \textbf{Initialize:} $z \gets rep(1,n)$
\State Evaluate $\ell^{(0)} = \log \mathcal{N}(\mu_{init} \mid \mu_0, \tau_0^2 I_d) + \text{Jointlik}(\data, z, \mu_{init},\Sigma)$
\For{$t = 1, \ldots, T$}
    \State \textbf{Propose:} $\mu^* \sim \mathcal{N}(\mu^{(t-1)}, \sigma^2 I_d)$
    \State \textbf{Evaluate target:}
    \Statex \hspace{2em} $\ell^* = \log \phi\left(\mu^*- \mu_0, \tau_0^2 I_d \right) + \text{Jointlik}(\data, z, \mu^*,\Sigma)$ 
    \State \textbf{Compute acceptance ratio:} $\alpha = \min(1, \exp(\ell^* - \ell^{(t-1)}))$
    \State \textbf{Accept/reject:} Draw $u \sim \text{Uniform}(0, 1)$
    \If{$u \leq \alpha$}
        \State $\mu^{(t)} \gets \mu^*$, $\ell^{(t)} \gets \ell^*$
    \Else
        \State $\mu^{(t)} \gets \mu^{(t-1)}$, $\ell^{(t)} \gets \ell^{(t-1)}$
    \EndIf
\EndFor
\State \Return $\mu^{(T)}$
\end{algorithmic}
\end{algorithm}

\vspace{1em}


\section{Proofs} \label{sec_proofs}

Here we restate the results in the main text and give their proofs. In addition to the notation used in the main text, we define:
\begin{itemize}
    \item $\logm(y)= \min(-\log(y), 0)$, for all $ y > 0$. Moreover, we will sometimes write $\logm^2(y)$ to mean $(\logm(y))^2$ 
    \item For $x \in \R^d$, $\|x\|_{\infty}=\max_{j} |x_j| $ and $\|x \|$ the Euclidean norm of $x$ and similarly with $\|x^{(m)}\|_{\infty}$ and $\|x^{(m)}\|$.
    \item For $f: \R^k \to \R$ bounded, $\| f \|_{\infty}=\sup_{x} |f(x)|$.
    \item To avoid an explosion of constants, we sometimes follow \cite{Fundamentals} and write $a \lesssim b$ to mean that $a \leq C b$ for some constant $0 < C < \infty$. We note that in the proof of Theorem \ref{Thm:DensityResult}, $C$ might depend on the dimension $d$. We write $a \asymp b$, if $a \lesssim b$ and $b \lesssim a$.
    \item We denote marginal distributions with a superscript, such as $p_{\theta^*}^{(m)}(x^{m})$, but we will write conditional densities without superscript, e.g, $p_{\theta^*}(x^{-(m)}\mid x^{(m)})$.
    \item We will abbreviate $\Prob(M=m\mid X=x)$ with $\Prob(M=m\mid x)$ and similarly for $\Prob(M=m\mid X^{(m)}=x^{(m)})$.
\end{itemize}

Before presenting our proofs, we quickly discuss an important implication of the MAR condition in Assumption \ref{asm_true_MDM}. As discussed in \cite{näf2024goodimputationmarmissingness} and also discussed in Section \ref{Sec_Motivating}, non-monotone MAR in dimensions $d > 2$ allows for almost arbitrary distribution shifts and can lead to identifiably issues. However, a particularly important implication of MAR, which will be the basis of much of our result, is that for any $\theta \in \Theta$,
\begin{align}\label{eq:importantMARimplication}
    \Prob(M=m \mid x^{(m)})&=\int \Prob(M=m \mid x) p_{\theta^*}(x^{-(m)} \mid x^{(m)})  d x^{-(m)} \nonumber \\
    &=\int \Prob(M=m \mid x) p_{\theta}(x^{-(m)} \mid x^{(m)})  d x^{-(m)}.
\end{align}
This is crucial, as without this condition, $\Prob(M=m \mid x^{(m)})$ depends on $\theta^*$, even if $\Prob(M=m \mid x)$ does not. In other words, even under our assumption that $\Prob(M=m \mid X)$ is fixed and not influenced by the choice of $\theta$, this will no longer be true for $\Prob(M=m \mid x^{(m)})$ under MNAR.

We start by defining a MAR-adapted version of the Hellinger distance by
\begin{align*}
    \HeMAR^2(p_{\theta_1}, p_{\theta_2})=\sum_{m} \int |\sqrt{p_{\theta_1}^{(m)}}(x^{(m)}) - \sqrt{p_{\theta_2}^{(m)}}(x^{(m)})|^2 \Prob(M=m \mid x^{(m)}) \mathrm{d}x^{(m)}. 
\end{align*}

Crucially, we can bound this as follows:

\begin{restatable}[Hellinger bounds]{prop}{Hellbound}\label{Hellingerbounding}
    Assume Assumption \ref{asm_true_MDM} holds true with $\Prob(M=0 \mid x) > \delta > 0$. Then for any $\theta_1, \theta_2 \in \Theta$
    \begin{align*}
    \delta \He^2(p_{\theta_1}, p_{\theta_2})\leq \HeMAR^2(p_{\theta_1}, p_{\theta_2})\leq \He^2(p_{\theta_1}, p_{\theta_2}).
\end{align*}
\end{restatable}

\begin{proof}
    We first note that by assumption,
    \begin{align*}
            \HeMAR^2(p_{\theta_1}, p_{\theta_2})&=\sum_{m} \int |\sqrt{p_{\theta_1}^{(m)}(x^{(m)})} - \sqrt{p_{\theta_2}^{(m)}(x^{(m)})}|^2 \Prob(M=m \mid x^{(m)}) \mathrm{d}x^{(m)}\\
            &\geq  \int |\sqrt{p_{\theta_1}(x)} - \sqrt{p_{\theta_2}(x)}|^2 \delta \mathrm{d}x^{(m)}\\
            &=  \delta \He^2(p_{\theta_1}, p_{\theta_2})
    \end{align*}

For the second inequality, we note that for any $m$,
\begin{align}\label{eq_fulldh}
    &\int |\sqrt{p_{\theta_1}^{(m)}(x^{(m)})} - \sqrt{p_{\theta_2}^{(m)}(x^{(m)})}|^2 \Prob(M=m \mid x^{(m)}) \mathrm{d}x^{(m)}\nonumber\\
    &=\int \left(p_{\theta_1}^{(m)}(x^{(m)}) + p_{\theta_2}^{(m)}(x^{(m)}) - 2\sqrt{p_{\theta_1}^{(m)}(x^{(m)})p_{\theta_2}^{(m)}(x^{(m)})} \right) \Prob(M=m \mid x^{(m)}) \mathrm{d}x^{(m)}.
\end{align}
In comparison,
\begin{align}\label{eq_missingdh}
&\int |\sqrt{p_{\theta_1}(x)} - \sqrt{p_{\theta_2}(x)}|^2  \Prob(M=m \mid x^{(m)}) \, \mathrm{d}x \nonumber \\
&= \int \left(p_{\theta_1}(x) + p_{\theta_2}(x) - 2\sqrt{p_{\theta_1}(x)p_{\theta_2}(x)}\right) \Prob(M=m \mid x^{(m)}) \, \mathrm{d}x  \nonumber \\
&= \int \Prob(M=m \mid x^{(m)}) \left(\int p_{\theta_1}(x) \, \mathrm{d}x^{(-m)} + \int p_{\theta_2}(x) \, \mathrm{d}x^{(-m)} - 2\int \sqrt{p_{\theta_1}(x)p_{\theta_2}(x)} \, \mathrm{d}x^{(-m)}\right) \mathrm{d}x^{(m)}\nonumber\\
&= \int  \left(p_{\theta_1}^{(m)}(x^{(m)}) + p_{\theta_2}^{(m)}(x^{(m)}) - 2\int\sqrt{p_{\theta_1}(x)p_{\theta_2}(x)} \, \mathrm{d}x^{(-m)}\right)\Prob(M=m \mid x^{(m)}) \mathrm{d}x^{(m)}.
\end{align}

By the Cauchy-Schwarz inequality applied to $\sqrt{p_{\theta_1}},\sqrt{p_{\theta_2}}$, for each fixed $x^{(m)}$:
\begin{align*}
    \sqrt{p_{\theta_1}^{(m)}(x^{(m)})p_{\theta_2}^{(m)}(x^{(m)})}&=\sqrt{\int p_{\theta_1}(x) \, \mathrm{d}x^{(-m)} \int p_{\theta_2}(x) \, \mathrm{d}x^{(-m)}}\\
    &\geq \int \sqrt{p_{\theta_1}(x)  p_{\theta_2}(x)} \, \mathrm{d}x^{(-m)},
\end{align*}
showing that
\begin{align*}
\int \sqrt{p_{\theta_1}^{(m)}(x^{(m)})p_{\theta_2}^{(m)}(x^{(m)})}\Prob(M=m \mid x^{(m)}) \mathrm{d}x^{(m)} \geq \int \int \sqrt{p_{\theta_1}(x)p_{\theta_2}(x)} \, \mathrm{d}x^{(-m)} \Prob(M=m \mid x^{(m)})  \mathrm{d}x^{(m)},
\end{align*}
and thus that for each $m$, \eqref{eq_fulldh} $\leq$ \eqref{eq_missingdh}. Since $\HeMAR^2(p_{\theta_1}, p_{\theta_2})$ is the sum of \eqref{eq_fulldh} over all $m$, and since, by MAR, $\Prob(M=m \mid x^{(m)})=\Prob(M=m \mid x)$, it follows that
\begin{align*}
   \HeMAR^2(p_{\theta_1}, p_{\theta_2}) &\leq \sum_{m} 
\int |\sqrt{p_{\theta_1}(x)} - \sqrt{p_{\theta_2}(x)}|^2  \Prob(M=m \mid x) \, \mathrm{d}x\\
&=\int |\sqrt{p_{\theta_1}(x)} - \sqrt{p_{\theta_2}(x)}|^2 \, \mathrm{d}x\\
&=\He^2(p_{\theta_1}, p_{\theta_2}).
\end{align*}

\end{proof}

\KLdef*

\begin{proof}
First, $\KLMAR(P_{\theta^*}\| P_{\theta})$ is clearly zero for $\theta=\theta^*$. Moreover, using the MAR nature of the missing mechanism, the inequality $\log(x)\leq x-1$, and basic algebra: 
\begin{align*}
    &\mathbb{E}_{(X,M)\sim\mathbb{P}_{\theta^*}}\left[ \log \left( \frac{p_{\theta}^{(M)}(X^{(M)})}{p_{\theta^*}^{(M)}(X^{(M)})} \right) \right] = \sum_{m}  \int_x  \log \left( \frac{p_{\theta}^{(m)}(x^{(m)})}{p_{\theta^*}^{(m)}(x^{(m)})}\right) p_{\theta^*}(x) \Prob(M=m \mid x) \mathrm{d}x \\
    & = \sum_{m}  \int_x  \log \left( \frac{p_{\theta}^{(m)}(x^{(m)})}{p_{\theta^*}^{(m)}(x^{(m)})}\right) p_{\theta^*}(x) \Prob(M=m \mid x^{(m)}) \mathrm{d}x \\
    & = \sum_{m}  \int_{x^{(m)}}  \log \left( \frac{p_{\theta}^{(m)}(x^{(m)})}{p_{\theta^*}^{(m)}(x^{(m)})}\right) \left( \int_{x^{(-m)}} p_{\theta^*}(x)  \mathrm{d}x^{(-m)} \right) \Prob(M=m \mid x^{(m)}) \mathrm{d}x^{(m)} \\
    & = \sum_{m}  \int_{x^{(m)}}  \log \left( \frac{p_{\theta}^{(m)}(x^{(m)})}{p_{\theta^*}^{(m)}(x^{(m)})}\right) p_{\theta^*}^{(m)}(x^{(m)}) \Prob(M=m \mid x^{(m)}) \mathrm{d}x^{(m)} \\
    & \leq \sum_{m}  \int_{x^{(m)}}  \left(\frac{p_{\theta}^{(m)}(x^{(m)})}{p_{\theta^*}^{(m)}(x^{(m)})}-1\right) p_{\theta^*}^{(m)}(x^{(m)}) \Prob(M=m \mid x^{(m)}) \mathrm{d}x^{(m)}\\
    & = \sum_{m}  \int_{x^{(m)}}  \frac{p_{\theta}^{(m)}(x^{(m)})}{p_{\theta^*}^{(m)}(x^{(m)})} p_{\theta^*}^{(m)}(x^{(m)}) \Prob(M=m \mid x^{(m)}) \mathrm{d}x^{(m)} - \sum_{m}  \int_{x^{(m)}}  p_{\theta^*}^{(m)}(x^{(m)}) \Prob(M=m \mid x^{(m)}) \mathrm{d}x \\
    &=\sum_{m}  \int_{x^{(m)}}  p_{\theta}^{(m)}(x^{(m)})  \Prob(M=m \mid x^{(m)}) \mathrm{d}x^{(m)} - \sum_{m}  \int_{x^{(m)}}  p_{\theta^*}^{(m)}(x^{(m)})  \Prob(M=m \mid x^{(m)}) \mathrm{d}x \\
    &=\sum_{m} \int_{x^{(m)}} \left( \int_{x^{(-m)}} p_{\theta}(x) dx^{(-m)} \right) \Prob(M=m \mid x^{(m)}) \mathrm{d}x^{(m)} \\
    & \hspace{5cm} - \sum_{m} \int_{x^{(m)}} \left( \int_{x^{(-m)}} p_{\theta^*}(x) dx^{(-m)} \right) \Prob(M=m \mid x^{(m)}) \mathrm{d}x^{(m)} \\
    &=\sum_{m}  \int_x  p_{\theta}(x)  \Prob(M=m \mid x^{(m)}) \mathrm{d}x - \sum_{m}  \int_x  p_{\theta^*}(x)  \Prob(M=m \mid x^{(m)}) \mathrm{d}x \\
    &=\sum_{m}  \int_x  p_{\theta}(x)  \Prob(M=m \mid x) \mathrm{d}x - \sum_{m}  \int_x  p_{\theta^*}(x)  \Prob(M=m \mid x) \mathrm{d}x \\
    &=1-1 \\
    &=0 \, ,
\end{align*}
which ends the proof.

\end{proof}

\KLloss*

\begin{proof}
$\KLMAR(P_{\theta^*}\| P_{\theta}) \geq 0$ has been proven in \Cref{definition_KL}. The show the other inequality $\KLMAR(P_{\theta^*}\| P_{\theta})\leq \KL(P_{\theta^*}\| P_{\theta})$, first write
    \begin{align*}
        \KLMAR&(P_{\theta^*}\| P_{\theta})=\E_{(X,M)\sim\mathbb{P}_{\theta^*}}\left[ \log \left( \frac{p_{\theta^*}^{(M)}(X^{(M)})}{p_{\theta}^{(M)}(X^{(M)})} \right) \right]\\
        &=\E_{(X,M)\sim\mathbb{P}_{\theta^*}}\left[ \log \left( \frac{p_{\theta^*}^{(M)}(X^{(M)})}{p_{\theta}^{(M)}(X^{(M)})} \right) \mathbbm{1}\left(M=0\right) + \log \left( \frac{p_{\theta^*}^{(M)}(X^{(M)})}{p_{\theta}^{(M)}(X^{(M)})} \right) \mathbbm{1}\left(M\ne0\right) \right] \\
        &=\E_{(X,M)\sim\mathbb{P}_{\theta^*}}\left[ \log \left( \frac{p_{\theta^*}(X)}{p_{\theta}(X)} \right) \mathbbm{1}\left(M=0\right) + \log \left(\frac{p_{\theta^*}(X)}{p_{\theta}(X)}\frac{p_{\theta}(X^{(-M)} \mid X^{(M)})}{p_{\theta^*}(X^{(-M)} \mid X^{(M)})} \right) \mathbbm{1}\left(M\ne0\right) \right] \\
        &=\E_{(X,M)\sim\mathbb{P}_{\theta^*}}\left[ \log \left( \frac{p_{\theta^*}(X)}{p_{\theta}(X)} \right) + \log \left(\frac{p_{\theta}(X^{(-M)} \mid X^{(M)})}{p_{\theta^*}(X^{(-M)} \mid X^{(M)})} \right) \mathbbm{1}\left(M\ne0\right) \right] \\
        &=\E_{(X,M)\sim\mathbb{P}_{\theta^*}}\left[ \log \left( \frac{p_{\theta^*}(X)}{p_{\theta}(X)} \right) \right]+ \E_{(X,M)\sim\mathbb{P}_{\theta^*}}\left[ \log \left( \frac{p_{\theta}(X^{(-M)} \mid X^{(M)})}{p_{\theta^*}(X^{(-M)} \mid X^{(M)})} \right) \mathbbm{1}\left(M\ne0\right) \right] \, .
    \end{align*}
    The first expectation in the sum is exactly $\KL(P_{\theta^*}\| P_{\theta})$, while we may write the second expectation as
\begin{align*}
    &\E_{(X,M)\sim\mathbb{P}_{\theta^*}}\left[ \log \left( \frac{p_{\theta}(X^{(-M)} \mid X^{(M)})}{p_{\theta^*}(X^{(-M)} \mid X^{(M)})} \right) \mathbbm{1}\left(M\ne0\right) \right]\\
    &=\sum_{m \neq 0} \int_x \log \left( \frac{p_{\theta}(x^{(-m)} \mid x^{(m)})}{p_{\theta^*}(x^{(-m)} \mid x^{(m)})} \right) p_{\theta^*}(x) \Prob(M=m \mid x) \mathrm{d}x\\
    &=\sum_{m \neq 0} \int_x \log \left( \frac{p_{\theta}(x^{(-m)} \mid x^{(m)})}{p_{\theta^*}(x^{(-m)} \mid x^{(m)})} \right) p_{\theta^*}(x^{(m)}) p_{\theta^*}(x^{(-m)} \mid x^{(m)}) \Prob(M=m \mid x) \mathrm{d}x\\
    &\leq\sum_{m \neq 0} \int_x \left( \frac{p_{\theta}(x^{(-m)} \mid x^{(m)})}{p_{\theta^*}(x^{(-m)} \mid x^{(m)})} - 1 \right) p_{\theta^*}(x^{(m)}) p_{\theta^*}(x^{(-m)} \mid x^{(m)}) \Prob(M=m \mid x) \mathrm{d}x\\
    &=\sum_{m \neq 0} \int_x p_{\theta^*}(x^{(m)}) p_{\theta}(x^{(-m)} \mid x^{(m)}) \Prob(M=m \mid x) \mathrm{d}x \\
    &\hspace{2cm} - \sum_{m \neq 0} \int_x p_{\theta^*}(x^{(m)}) p_{\theta^*}(x^{(-m)} \mid x^{(m)}) \Prob(M=m \mid x) \mathrm{d}x \\
    &=\sum_{m \neq 0} \int_x p_{\theta^*}(x^{(m)}) p_{\theta}(x^{(-m)} \mid x^{(m)}) \Prob(M=m \mid x^{(m)}) \mathrm{d}x \\
    &\hspace{2cm} - \sum_{m \neq 0} \int_x p_{\theta^*}(x^{(m)}) p_{\theta^*}(x^{(-m)} \mid x^{(m)}) \Prob(M=m \mid x^{(m)}) \mathrm{d}x \\
    &=\sum_{m \neq 0} \int_{x^{(m)}} p_{\theta^*}(x^{(m)}) \left( \int_{x^{(-m)})} p_{\theta}(x^{(-m)} \mid x^{(m)}) \mathrm{d}x^{(-m)} \right) \Prob(M=m \mid x^{(m)})\mathrm{d}x^{(m)} \\
    &\hspace{2cm} - \sum_{m \neq 0} \int_{x^{(m)}} p_{\theta^*}(x^{(m)}) \left( \int_{x^{(-m)})} p_{\theta^*}(x^{(-m)} \mid x^{(m)}) \mathrm{d}x^{(-m)} \right) \Prob(M=m \mid x^{(m)}) \mathrm{d}x^{(m)} \\
    &=\sum_{m \neq 0} \int_{x^{(m)}} p_{\theta^*}(x^{(m)}) \Prob(M=m \mid x^{(m)}) \mathrm{d}x^{(m)} - \sum_{m \neq 0} \int_{x^{(m)}} p_{\theta^*}(x^{(m)}) \Prob(M=m \mid x^{(m)}) \mathrm{d}x^{(m)} \\
    &=0 \, .
\end{align*}
Thus by the same arguments as in \Cref{definition_KL}, the second expectation is smaller than or equal to zero, leading to $\KLMAR(P_{\theta^*}\| P_{\theta}) \leq  \KL(P_{\theta^*}\| P_{\theta})$. 

\vspace{0.2cm}
Regarding the implication of $\KLMAR(P_{\theta^*} \| P_{\theta}) > 0$ by $\KL(P_{\theta^*} \| P_{\theta}) > 0$ and $\Prob(M=0 \mid X=x) > 0$, notice that if $\KL(P_{\theta^*} \| P_{\theta}) > 0$, then there exists a set $A$ with $P_{\theta^*}(A) > 0$, such that $\log \left( \frac{p_{\theta^*}(x) }{p_{\theta}(x)}\right) > 0$. 
Then, we have the strict inequality $\log\left(\frac{p_{\theta}(x)}{p_{\theta^*}(x)}\right) < \frac{p_{\theta}(x)}{p_{\theta^*}(x)}-1$ for $x \in A$ (since ${p_{\theta^*}(x)}\ne{p_{\theta}(x)}$ on $A$), which along with $\Prob(M=0 \mid X=x) > 0$ for all $x \in \mathcal{X}$ and $P_{\theta^*}(A) > 0$ implies:
$$
\int_{x\in A} \log \left( \frac{p_{\theta}(x)}{p_{\theta^*}(x)}\right) \Prob(M=0 \mid X=x) p_{\theta^*}(x) \mathrm{d}x < \int_{x\in A} \left( \frac{p_{\theta}(x)}{p_{\theta^*}(x)}-1\right) \Prob(M=0 \mid X=x) p_{\theta^*}(x) \mathrm{d}x \, .
$$
Hence, following the same arguments as in \Cref{definition_KL},
\begin{align*}
    &\mathbb{E}_{(X,M)\sim\mathbb{P}_{\theta^*}}\left[ \log \left( \frac{p_{\theta}^{(M)}(X^{(M)})}{p_{\theta^*}^{(M)}(X^{(M)})} \right) \right] \\
    & = \sum_{m}  \int_{x^{(m)}}  \log \left( \frac{p_{\theta}^{(m)}(x^{(m)})}{p_{\theta^*}^{(m)}(x^{(m)})}\right) p_{\theta^*}^{(m)}(x^{(m)}) \Prob(M=m \mid x^{(m)}) \mathrm{d}x^{(m)} \\
    & = \int_{x}  \log \left( \frac{p_{\theta}(x)}{p_{\theta^*}(x)}\right) p_{\theta^*}(x) \Prob(M=0 \mid X=x) \mathrm{d}x \\
    & \hspace{5cm} + \sum_{m\ne0}  \int_{x^{(m)}}  \log \left( \frac{p_{\theta}^{(m)}(x^{(m)})}{p_{\theta^*}^{(m)}(x^{(m)})}\right) p_{\theta^*}^{(m)}(x^{(m)}) \Prob(M=m \mid x^{(m)}) \mathrm{d}x^{(m)} \\
    & = \int_{x\in A} \log \left( \frac{p_{\theta}(x)}{p_{\theta^*}(x)}\right) p_{\theta^*}(x) \Prob(M=0 \mid x) \mathrm{d}x + \int_{x\notin A} \log \left( \frac{p_{\theta}(x)}{p_{\theta^*}(x)}\right) p_{\theta^*}(x) \Prob(M=0 \mid x) \mathrm{d}x \\
    & \hspace{5cm} + \sum_{m\ne0}  \int_{x^{(m)}}  \log \left( \frac{p_{\theta}^{(m)}(x^{(m)})}{p_{\theta^*}^{(m)}(x^{(m)})}\right) p_{\theta^*}^{(m)}(x^{(m)}) \Prob(M=m \mid x^{(m)}) \mathrm{d}x^{(m)} \\
    & < \int_{x\in A} \left( \frac{p_{\theta}(x)}{p_{\theta^*}(x)}-1\right) p_{\theta^*}(x) \Prob(M=0 \mid x) \mathrm{d}x + \int_{x\notin A}  \left( \frac{p_{\theta}(x)}{p_{\theta^*}(x)}-1\right) p_{\theta^*}(x) \Prob(M=0 \mid x) \mathrm{d}x \\
    & \hspace{5cm} + \sum_{m\ne0}  \int_{x^{(m)}} \left( \frac{p_{\theta}^{(m)}(x^{(m)})}{p_{\theta^*}^{(m)}(x^{(m)})}-1\right) p_{\theta^*}^{(m)}(x^{(m)}) \Prob(M=m \mid x^{(m)}) \mathrm{d}x^{(m)} \\
    & = \sum_{m}  \int_{x^{(m)}} \left( \frac{p_{\theta}^{(m)}(x^{(m)})}{p_{\theta^*}^{(m)}(x^{(m)})}-1\right) p_{\theta^*}^{(m)}(x^{(m)}) \Prob(M=m \mid x^{(m)}) \mathrm{d}x^{(m)} \\
    &=0 \, ,
\end{align*}
which shows that $\KLMAR(P_{\theta^*} \| P_{\theta}) > 0$. Finally, assume that $\Prob(M=0 \mid X=x)=0$ for $x \in A$ with $\Pjoint(A) > 0$. Then it is possible to construct a density $p_{\theta_1}$ that agrees with $p_{\theta^*}$ on $A^c$ and on all $k < d$ dimensional marginals, but for which $p_{\theta_1}(x)\neq p_{\theta^*}(x)$. But this difference is masked by $\PMzerox=0$, so that $\KLMAR(P_{\theta^*} \| P_{\theta_1}) = 0$.
\end{proof}

\Generalrates*

\begin{proof}
\bigskip

We will actually show that,
\begin{align}\label{eq_actualresult}
    \E_{\data \sim \Pthetajointn}\left[\Pi\left(d(\theta , \theta^*) \geq \CIV \varepsilon_n \, \big| \, X_1^{(M_1)}, \dots, X_n^{(M_n)}\right)\right] \to 0,
\end{align}
which implies the result.

\noindent \textit{Step 1: Extending the Testing Condition.} We start by combining Assumption \ref{asm_test2}, in combination with \ref{iii_coveringnumber} to achieve a more directly usable testing condition on $\Theta_n$: There exist $\CIV=\CIV(a,\CIII, \CI)$ and tests $\psi_n=\psi_n(\data)$ such that
\begin{align}\label{eq_psitests}
    \E_{S_n \sim \mathbb{P}^n_{\theta^*}}[\psi_n(\data)] \to 0 \quad \quad \sup_{\{\theta \in \Theta_n: d(\theta, \theta^*) > \CIV \varepsilon_n\}} \E_{S_n \sim \mathbb{P}^n_{\theta}}[1-\psi_n(\data)] \leq e^{-(\CI+4) n\varepsilon_n^2}.
\end{align}

To show this, we first define the sets
\[
A_j= \{ \theta \in  \Theta_n: d(\theta, \theta^*): \CIV j \varepsilon_n < d(\theta, \theta^*) \leq \CIV(j+1)  \varepsilon_n  \},
\]
such that 
\[
\{\theta \in  \Theta_n: d(\theta, \theta^*) > 2\CIV \varepsilon_n \}=\bigcup_{j=1}^{\infty} A_j.
\]
For a fixed but arbitrary $j\geq 1$, let now $(B_{l,j})_{l=1}$ be a minimal covering with radius $a (\CIV j \varepsilon_n)$, with $a$ taken from Assumption \ref{asm_test2}. Let $b_{j, l}$, $l=1, \ldots, {N(a \CIV j \varepsilon_n, A_j, d)}$ be the center of these balls. Then, note that $A_j$ and the covering ball $B_{l,j}$ must overlap, as otherwise we could remove $B_{l,j}$ and would not have a minimial covering. Taking, any $\theta \in A_j \cap B_{l,j}$, we then have:
\begin{align*}
   \CIV j \varepsilon_n < d(\theta, \theta^*) \leq  d(\theta^*, g_{j,l}) + d(g_{j,l}, \theta)  \leq d(\theta^*, g_{j,l}) + a \CIV j \varepsilon_n,
\end{align*}
or 
\begin{align}
    d(\theta^*, g_{j,l}) \geq   \CIV j \varepsilon_n - a \CIV j \varepsilon_n > \CIV j \varepsilon_n ,
\end{align}
since $a < 1$. By Assumption \ref{asm_test2}, with $\varepsilon$ taken to be $\CIV j \varepsilon_n$ and $\theta_1=g_{j,l}$, there exists a test $\phi_n^{j,l}$ such that 
\[
\mathbb{E}_{\mathcal{S}_n\sim\mathbb{P}_{\theta^*}^n}\left[\phi_n^{j,l}(\data)\right] \leq  e^{-\CIII n (\CIV j \varepsilon_n)^2} \quad , \quad  \sup_{\theta \in \Theta_n: d(\theta, g_{j,l} ) < a  \CIV j \varepsilon_n}\mathbb{E}_{\mathcal{S}_n\sim\mathbb{P}_{\theta}^n}\left[1-\phi_n^{j,l}(\data)\right] \leq e^{-\CIII n (\CIV j \varepsilon_n)^2}.
\]
We take 
\begin{align}\label{eq_newtest}
    \psi_n(\data)= \sup_{j,l} \phi_n^{j,l}(\data),
\end{align}
and note that choosing $\CIV \geq  1/a$ and since $ j \geq 1$ and $A_j \subset \Theta_n$
\begin{align*} 
    N(a \CIV j \varepsilon_n, A_j, d) \leq N( \varepsilon_n, A_j, d)\leq N( \varepsilon_n, \Theta_n, d) \leq e^{\CII n \varepsilon_n^2},
\end{align*}
where the last step followed because of Assumption \ref{iii_coveringnumber}. Then for the new test we have
\begin{align*}
   \E_{\data \sim \Pjointn} [\psi_n(\data)] &\leq \sum_{l=1}^{ N(a \CIV j \varepsilon_n, A_j, d) } \sum_{j=1}^{\infty} \E_{\data \sim \Pjointn} [\phi^{l,j}_n]\\
   &\leq  \sum_{j=1}^{\infty} N(a \CIV j \varepsilon_n, A_j, d) e^{-\CIII n (\CIV j \varepsilon_n)^2}\\
   &\leq  e^{\CII n \varepsilon_n^2} \sum_{j=1}^{\infty}  (e^{-\CIII n (\CIV \varepsilon_n)^2})^j\\
   &=  e^{\CII n \varepsilon_n^2} \frac{e^{-\CIII n (\CIV  \varepsilon_n)^2}}{1-e^{-\CIII n (\CIV  \varepsilon_n)^2}}\\
   &\leq \CI e^{(\CII-\CIII \CIV^2) n \varepsilon_n^2},
\end{align*}
which goes to zero for, $\CIV^2 > \CII/\CIII$. Moreover, 
\begin{align*}
   \sup_{\theta \in  \Theta_n: d(\theta, \theta^*) > \CIV \varepsilon_n }\E_{\data \sim \Pthetajointn} [1-\psi_n(\data)] &\leq \sup_{j,l} \sup_{\theta \in B_{j,l}} \E_{\data \sim \Pthetajointn} [1-\psi_n(\data)]\\
   &\leq \sup_{j,l} e^{-\CIII n (\CIV j \varepsilon_n)^2}\\
   &\leq \sup_{j,l} e^{-\CIII n (\CIV \varepsilon_n)^2}\\
   &\leq \sup_{j,l} e^{-(\CI+4)\varepsilon_n^2},
\end{align*}
for $\CIV^2 > (\CI+4)/\CIII$. Thus, taking $\CIV > \max(1/a,\sqrt{\CII/\CIII}, \sqrt{(\CI+4)/\CIII})$ gives the result.



\vspace{0.2cm}
\noindent \textit{Step 2: } 
Define the set
\begin{align*}
    \mathcal{C}_n= \{\theta \in \Theta: d(\theta, \theta^*) \geq \CIV \varepsilon_n\} 
\end{align*}
Moreover, to make use of Assumption \ref{v_priormassII}, we define:
\begin{align}\label{Omegaeq2}
    \Omega_n(r) = \left\{\int \prod_{i=1}^n \frac{p_{\theta}^{(M_i)}\left(X_i^{(M_i)}\right)}{p_{\theta^*}^{(M_i)}\left(X_i^{(M_i)}\right)} d\Pi(\theta) \geq \Pi(\mathcal{B}(\theta^*,r; \mathbb{P}_{\theta^*})) e^{-2 n r^2} \right\} .
\end{align}

We now use the tests in \eqref{eq_psitests} from Step 1:
\begin{align*}
    &\E_{\data \sim \Pjointn}\left[\Pi\left(\mathcal{C}_n \, \big| \data \right)\right]\\
    &=\E_{\data \sim \Pjointn}\left[\psi_n(\data)\Pi\left(\mathcal{C}_n \, \big| \, \data\right)\right] + \E_{\data \sim \Pjointn}\left[(1-\psi_n(\data))\Pi\left(\mathcal{C}_n \, \big| \, \data\right)\right]\\
    &\leq \underbrace{\E_{\data \sim \Pjointn}\left[\psi_n(\data)\right]}_{\to 0} + \E_{\data \sim \Pjointn}\left[(1-\psi_n(\data))\Pi\left(\mathcal{C}_n \, \big| \, \data\right)\right].
\end{align*}
For the second element in the previous sum:
\begin{align*}
    &\E_{\data \sim \Pjointn}\left[(1-\psi_n(\data))\Pi\left(\mathcal{C}_n \, \big| \, \data\right)\right]\\
    &=\E_{\data \sim \Pjointn}\left[(1-\psi_n(\data))\Pi\left(\mathcal{C}_n \, \big| \, \data\right) \mathbf{1}_{\Omega_n(\bar{\varepsilon}_n)}\right] + \E_{\data \sim \Pjointn}\left[(1-\psi_n(\data))\Pi\left(\mathcal{C}_n \, \big| \, \data\right)\mathbf{1}_{\Omega_n(\bar{\varepsilon}_n)^c}\right]\\
       &\leq \E_{\data \sim \Pjointn}\left[(1-\psi_n(\data))\Pi\left(\mathcal{C}_n \, \big| \, \data\right) \mathbf{1}_{\Omega_n(\bar{\varepsilon}_n)}\right] +  \underbrace{\Pjointn(\Omega_n(\bar{\varepsilon}_n)^c)}_{\to 0},
\end{align*}
where the convergence to zero follows from Lemma \ref{Omegaanbound} using $r=\bar{\varepsilon}_n$ and the fact that $n \bar{\varepsilon}_n^2 \to \infty$. Finally, for the last expectation, we have, using the definition of $\Omega_n(\bar{\varepsilon}_n)$, as well as Assumption \ref{v_priormassII}),
\begin{align*}
    &\E_{\data \sim \Pjointn}\left[(1-\psi_n(\data))\Pi\left(\mathcal{C}_n \, \big| \, \data\right) \mathbf{1}_{\Omega_n(\bar{\varepsilon}_n)}\right] \\
    &\leq \frac{e^{2n \bar{\varepsilon}_n^2}}{\Pi(\mathcal{B}(\theta^*,\bar{\varepsilon}_n; \mathbb{P}_{\theta^*}))}\E_{\data \sim \Pjointn}\left[(1-\psi_n(\data)) \mathbf{1}_{\Omega_n(\bar{\varepsilon}_n)}\int_{\mathcal{C}_n}  \prod_{i=1}^n \frac{p_{\theta}^{(M_i)}\left(X_i^{(M_i)}\right)}{p_{\theta^*}^{(M_i)}\left(X_i^{(M_i)}\right)} d\Pi(\theta)\right] \\
& \leq    \frac{e^{2 n \bar{\varepsilon}_n^2 }}{\Pi(\mathcal{B}(\theta^*,\bar{\varepsilon}_n; \mathbb{P}_{\theta^*}))}\int_{\mathcal{C}_n}  \E_{\data \sim \Pthetajointn}\left[(1-\psi_n(\data)) \right] d\Pi(\theta) \\
&\leq e^{n \bar{\varepsilon}_n^2 (2+\CI)}\int_{\mathcal{C}_n}  \E_{\data \sim \Pthetajointn}\left[(1-\psi_n(\data)) \right] d\Pi(\theta)\\
&=e^{n \bar{\varepsilon}_n^2 (2+\CI)}\int_{\mathcal{C}_n \cap \Theta_n}  \E_{\data \sim \Pthetajointn}\left[(1-\psi_n(\data)) \right] d\Pi(\theta) + e^{n \bar{\varepsilon}_n^2 (2+\CI)}\int_{\mathcal{C}_n \cap \Theta_n^c}  \E_{\data \sim \Pthetajointn}\left[(1-\psi_n(\data)) \right] d\Pi(\theta)\\
& \leq e^{n \bar{\varepsilon}_n^2 (2+\CI)}\int_{\mathcal{C}_n \cap \Theta_n}  \E_{\data \sim \Pthetajointn}\left[(1-\psi_n(\data)) \right] d\Pi(\theta) + e^{n \bar{\varepsilon}_n^2 (2+\CI)} \Pi(\Theta_n^c).
\end{align*}
For the second element of the sum, we obtain $e^{n \bar{\varepsilon}_n^2 (2+\CI)} \Pi(\Theta_n^c) \leq e^{n \bar{\varepsilon}_n^2 (2+\CI)} e^{-n \bar{\varepsilon}_n^2(4+\CI)} \to 0$ by Assumption \ref{iv_priormassI}. Similarly, for the first element of the sum, we get 
\begin{align*}
    e^{-n \bar{\varepsilon}_n^2 (2+\CI)}\int_{\mathcal{C}_n \cap \Theta_n}  \E_{\data \sim \Pthetajointn}\left[(1-\psi_n(\data)) \right] d\Pi(\theta) \leq e^{-n \bar{\varepsilon}_n^2 (2+\CI)} e^{-(\CI+4) n\varepsilon_n^2} \to 0,
\end{align*}
by \eqref{eq_psitests}. Thus $\E_{\data \sim \Pjointn}\left[(1-\psi_n(\data))\Pi\left(\mathcal{C}_n \, \big| \, \data\right) \mathbf{1}_{\Omega_n(\bar{\varepsilon}_n)}\right] \to 0$, and we obtain the result.


\end{proof}

\Densityrates*

\begin{proof}
We largely follow the arguments in \cite[Chapter 9.4]{Fundamentals} and study the following sieve space:
\begin{align*}
    \Theta_n=\{p_{F, \Sigma}: F \in \mathcal{F}_{\NI, \varepsilon,a}, S \in \mathcal{D}_{\sigma, \varepsilon} \}
\end{align*}
with $\varepsilon, a, \sigma$ positive and $\NI$ integer,
\begin{align*}
    \mathcal{F}_{\NI, \varepsilon,a}&=\left\{ \sum_{j=1}^{\infty} w_j \delta_{z_j}: \sum_{j=\NI+1}^{\infty} w_j < \varepsilon^2, z_1, \ldots, z_{\NI} \in [-a,a]^d \right\}.\\
    \mathcal{D}_{\sigma, \varepsilon}&=\{\Sigma: \sigma^2 \leq \lambda_1(\Sigma) \leq \lambda_d(\Sigma) < \sigma^2 (1+ \varepsilon^2)^n\},
\end{align*}
where $\lambda_1(\Sigma), \ldots, \lambda_d(\Sigma)$ denote the eigenvalues of a matrix $\Sigma$ ordered by size. In the following, $\sigma, \varepsilon$ will decrease at rates specified below, while $a$ and $\NI$ are intended to increase.

Now, to prove \ref{v_priormassII}, we use the construction in the proof of (ii) of Proposition 9.14 in \cite{Fundamentals} to obtain
\begin{itemize}
    \item a partition $\{U_1, \ldots, U_{\NI'}\}$, with $\NI' \lesssim  \NI$ of $[-a_{\sigma}, a_{\sigma}]^d$ of diameter at most $\sigma$.
    \item a partition $U_1, \ldots, U_{\NII}$ of $\R^d$ of size $\NII \lesssim  \NI$, with the property $(\sigma \varepsilon^2)^d \lesssim  \alpha(U_j) \lesssim 1$,
\end{itemize}
where $a_{\sigma}=\alpha_0 \logm(\sigma)^{1/\tau}$ for some $\alpha_0$ large. Given the number of sets $\NII$, we then construct a set $B_{\varepsilon, \NII, \sigma, \beta} \subset \Theta$, such that for some $C >0$,
\[
B_{\varepsilon, \NII, \sigma, \beta} \subset \left\{(F, \Sigma): \He^2(p_{\theta^*}, p_{F,\Sigma}) \lesssim \sigma^{2\beta} + \varepsilon^2 \right\},
\]
and for small enough, $\sigma, \varepsilon$
\begin{align}\label{eq_Blowerbound}
    \Pi(B_{\varepsilon, \NII, \sigma, \beta}) \geq e^{-c (\logm(\sigma))^{d/\tau} \sigma^{-d} (\logm(\varepsilon))^{d+1} + \sigma^{-\kappa}}
\end{align}
exactly as in Equation (9.14) of \cite[Chapter 9]{Fundamentals}. Crucially, \cite{Fundamentals} now show that
\begin{align*}
    &\left\{(F, \Sigma): \He^2(p_{\theta^*}, p_{F,\Sigma}) \lesssim \sigma^{2\beta} + \varepsilon^2 \right\} \\
    &\subset \left \{ \theta: \E_{X \sim P_{\theta^*}}\left[\log\left(  \frac{p_{\theta^*}(X)}{p_{\theta}(X)}\right)\right] \leq \varepsilon_n^2, \E_{X \sim P_{\theta^*}}\left[\log\left(  \frac{p_{\theta^*}(X)}{p_{\theta}(X)}\right)^2\right] \leq \varepsilon_n^2 \right \},
\end{align*}
 for $\sigma, \varepsilon$ chosen such that
 \[
 (\sigma^{2\beta} + \varepsilon^2) \left[ \logm\left(\frac{\varepsilon}{\sigma^d}\right)  \right]^2 \lesssim \varepsilon_n^2, \frac{\varepsilon^4}{\sigma^d} < 1, (\logm(\sigma))^{d/\tau} \sigma^{-d} (\logm(\varepsilon))^{d+1} + \sigma^{-\kappa} \leq n\varepsilon_n^2.
 \]
We now adapt this argument for our purposes. Since the diameter of each $U_{j}$ in the partition $\{U_1, \ldots, U_{\NI'}\}$ of $[-a_{\sigma}, a_{\sigma}]^d $ is at most $\sigma$, it can be shown that for any $x \in [-a_{\sigma}, a_{\sigma}]^d$ and $(F, \Sigma) \in B_{\varepsilon, \NII, \sigma, \beta}$,
\begin{align}\label{eq_lowerbound_1}
    \frac{1}{p_{F,\Sigma}(x)}  \leq \frac{\sigma^d}{e^{-1}} \frac{1}{F(U_{j(x))})}\leq \frac{\sigma^d}{e^{-1} \varepsilon^4},
\end{align}
by the construction of $B_{\varepsilon, \NII, \sigma, \beta}$. Moreover, for any $x \notin [-a_{\sigma}, a_{\sigma}]^d$, it can be shown that
\begin{align}\label{eq_lowerbound_2}
    \frac{1}{p_{F,\Sigma}(x)} \lesssim \sigma^d e^{2d \|x\|^2/\sigma^2}
\end{align}

Now for $(F, \Sigma) \in B_{\varepsilon, \NII, \sigma, \beta}$ construct the new set of measures $p_{m,\theta^*,F,\Sigma}(x)=p_{\theta^*}^{(m)}(x^{(m)})p_{F,\Sigma}(x^{(-m)} \mid x^{(m)})$. We first note that these are well-defined densities a.s., since $p_{F,\Sigma}$ is supported on all of $\R^d$. Moreover, we have that for all $m \neq 1$,  $d^{(-m)}=\sum_{j=1}^{d} m_j \leq d-1$, and $\sigma < 1$ small,
\begin{align*}
    p_{m,\theta^*,F,\Sigma}(x) \leq \frac{2\|p_{\theta^*}^{(m)} \|_{\infty}}{\sigma^{d^{(-m)}}} \leq \frac{2\max_{m}\|p_{\theta^*}^{(m)} \|_{\infty}}{\sigma^{d-1}} \text{ for all } x \in \mathcal{X}, m\in\{0,1\}^d\setminus \{1\}.
\end{align*}
Indeed, for $(F, \Sigma) \in B_{\varepsilon, \NII, \sigma, \beta}$, we have $p_{F,\Sigma}( x)=\int \phi(x-z, \Sigma) d F(z) $, so that by the bound in \Cref{lem_Gaussianbound}, for all $m \neq 1$,
\begin{align*}
    \sup_{x \in \mathcal{X}} p_{F,\Sigma}(x^{(-m)} \mid x^{(m)}) \leq \frac{1}{(2 \pi \lambda_{1}(\Sigma))^{d^{(-m)}/2} }.
\end{align*}
Since also, 
\begin{align*}
    \frac{1}{\sigma^2} \leq \lambda_1(\Sigma^{-1}) \leq \lambda_d(\Sigma^{-1}) \leq \frac{1+\sigma^{\beta}}{\sigma^2},
\end{align*}
we have that, for $m\neq 1$,
\begin{align*}
   \sup_{x \in \R^d} p_{F,\Sigma}(x^{(-m)} \mid x^{(m)})\leq \frac{1}{(2 \pi \lambda_{1}(\Sigma))^{d^{(-m)}/2} } \lesssim \left(\frac{1+\sigma^{\beta}}{\sigma^2}\right)^{d^{(-m)}/2} \leq \frac{2}{\sigma^{d-1}},
\end{align*}
for $\sigma < 1$ small.

Then, from \eqref{eq_priorboundH2} in \Cref{Lem_Priorboundlemmas}, for $\tilde{\varepsilon}<0.4$,
\begin{align}\label{eq_Lemmapp}
    &\E_{(X,M) \sim \Pjoint}\left[ \log \left( \frac{p_{\theta^*}^{(M)}(X^{(M)})}{p_{F,\Sigma}^{(M)}(X^{(M)})} \right)^2 \right] \nonumber \\
   &\leq \HeMAR^2(p_{\theta^*},p_{F,\Sigma})(12 + 2\logm^2(\tilde{\varepsilon})) \nonumber \\ 
    &+ 8\E_{(X,M)\sim \Pjoint}\!\left[\log\left(\frac{p_{\theta^*}^{(M)}(X^{(M)})}{p_{F,\Sigma}^{(M)}(X^{(M)})}\right)^2\mathbf{1}\left \{\frac{p_{F,\Sigma}^{(M)}(X^{(M)})}{p_{\theta^*}^{(M)}(X^{(M)})} \leq \tilde{\varepsilon} \right \}\right].
\end{align}
We choose
\begin{equation}\label{eq_tildeeps}
    \tilde{\varepsilon}=\frac{e^{-1} \varepsilon^4}{\max_m \|p^{(m)}_{\theta^*}\|_{\infty} \sigma},
\end{equation}
for $\varepsilon$ small enough. If, $\He^2(p_{\theta^*}, p_{F,\Sigma}) \lesssim \sigma^{2\beta} + \varepsilon^2$, this leads to the first part of the above to be bounded as:
\begin{align}\label{First_bound}
    \HeMAR^2(p_{\theta^*},p_{F,\Sigma})(12 + 2\logm^2\tilde{\varepsilon}) \lesssim (\sigma^{2\beta} + \varepsilon^2) \logm^2\left(\frac{\varepsilon^4}{\sigma}\right).
\end{align}
Since, for all $m$
\begin{align*}
    \frac{p_{\theta^*}^{(m)}(x^{(m)})}{p_{F,\Sigma}^{(m)}(x^{(m)})} = \frac{p_{m,\theta^*,F,\Sigma}(x)}{p_{F,\Sigma}(x)},
\end{align*}
we note that from \eqref{eq_lowerbound_1}, 
\begin{align*}
  \|x \|_{\infty} \leq a_{\sigma} \implies  \frac{p_{\theta^*}^{(m)}(x^{(m)})}{p_{F,\Sigma}^{(m)}(x^{(m)})} \leq \frac{\sigma^d p_{m,\theta^*,F,\Sigma}(x)}{e^{-1} \varepsilon^4}\leq \frac{\sigma \max_m \|p^{(m)}\|_{\infty}}{e^{-1} \varepsilon^4}.
\end{align*}
In reverse, this means that
\begin{align*}
    \frac{p_{\theta^*}^{(m)}(x^{(m)})}{p_{F,\Sigma}^{(m)}(x^{(m)})}  > \tilde{\varepsilon}^{-1} =\frac{\sigma \max_m \|p^{(m)}\|_{\infty}}{e^{-1} \varepsilon^4} \implies \|x \|_{\infty} > a_{\sigma} \implies \log\left(\frac{1}{p_{F,\Sigma}(x)}\right) \lesssim  \logm(\sigma) +\frac{\|x\|^2}{\sigma^2},
\end{align*}
from \eqref{eq_lowerbound_2}. Moreover, we have for each $m\in \{0,1\}^d\setminus\{1\}$ and $x$ such that
\[
\frac{p_{m,\theta^*,F,\Sigma}(x)}{p_{F,\Sigma}(x)}=\frac{p_{\theta^*}^{(m)}(x^{(m)})}{p_{F,\Sigma}^{(m)}(x^{(m)})} > \tilde{\varepsilon}^{-1},
\]
that
\begin{align*}
    \log\left(\frac{p_{m,\theta^*,F,\Sigma}(x)}{p_{F,\Sigma}(x)}\right)^2 &= \log(p_{m,\theta^*,F,\Sigma}(x))^2+\log\left( \frac{1}{p_{F,\Sigma}(x)}\right)^2 +\log(p_{m,\theta^*,F,\Sigma}(x))\log\left( \frac{1}{p_{F,\Sigma}(x)}\right)\\
    &=\log\left( \frac{1}{p_{F,\Sigma}(x)}\right)^2 \left( \frac{\log(p_{m,\theta^*,F,\Sigma}(x))}{\log\left( \frac{1}{p_{F,\Sigma}(x)}\right)} + 1\right)^2\\
    &\lesssim \left(\logm(\sigma) +\frac{\|x\|^2}{\sigma^2}\right)^2 \left( \frac{\log(2\max_{m}\|p_{\theta^*}^{(m)} \|_{\infty}) + (d-1)\logm(\sigma)}{\logm(\sigma) +\frac{\|x\|^2}{\sigma^2}} + 1\right)^2\\
    &\lesssim \left(\logm(\sigma) +\frac{\|x\|^2}{\sigma^2}\right)^2,
\end{align*}
where the last step followed because, for $\sigma < 1 $ small,
\begin{align*}
    \frac{\log(2\max_{m}\|p_{\theta^*}^{(m)} \|_{\infty}) + (d-1)\logm(\sigma)}{\logm(\sigma) +\frac{\|x\|^2}{\sigma^2}} \leq C,
\end{align*}
for some constant $C$. Thus, for the second part,
\begin{align*}
    &\E_{(X,M)\sim \Pjoint}\!\left[\log\left(\frac{p_{\theta^*}^{(M)}(X^{(M)})}{p_{F,\Sigma}^{(M)}(X^{(M)})}\right)^2\mathbf{1}\left \{\frac{p_{F,\Sigma}^{(M)}(X^{(M)})}{p_{\theta^*}^{(M)}(X^{(M)})} \leq \tilde{\varepsilon} \right \}\right]\\
    &=\sum_{m} \int \log\left(\frac{p_{m,\theta^*,F,\Sigma}(x)}{p_{F,\Sigma}(x)}\right)^2\mathbf{1}\left \{\frac{ p_{m,\theta^*,F,\Sigma}(x)}{p_{F,\Sigma}(x)} \geq \tilde{\varepsilon}^{-1} \right \}p_{\theta^*}(x) \Prob(M=m\mid x) \mathrm{d}x\\
    &\lesssim  \int_{\|x \|_{\infty} > a_{\sigma}} \left(\logm(\sigma)^2  +  \frac{\|x\|^4}{\sigma^4} + 2 \logm(\sigma) \frac{\|x\|^2}{\sigma^2}\right)p_{\theta^*}(x)  \mathrm{d}x,
\end{align*}
again using the fact that $\sum_{m} \Prob(M=m \mid x )=1$. Now by the tail assumption on $p_{\theta^*}$, the fact that $\|x \| \geq \|x \|_{\infty}$ and the choice of $a_{\sigma}=a_0 \logm(\sigma)^{1/\tau}$, for large enough $a_0$, 
\begin{align*}
     &\leq \int_{\|x \| > a_{\sigma}} \left(\logm(\sigma)^2  +  \frac{\|x\|^4}{\sigma^4} + 2 \logm(\sigma) \frac{\|x\|^2}{\sigma^2}\right)c e^{-b \|x\|^{\tau}}  \mathrm{d}x\\
          &= \frac{e^{-b a_{\sigma}^{\tau}}}{\sigma^4}\int_{\|x \| > a_{\sigma}} \left(\sigma^4\logm(\sigma)^2  +  \|x\|^4 + 2 \logm(\sigma) \sigma^2\|x\|^2\right)c e^{-b (\|x\|^{\tau} -a_{\sigma}^{\tau})}  \mathrm{d}x \\
          &\lesssim \frac{e^{-b a_{\sigma}^{\tau}}}{\sigma^4}\\
          &=\sigma^{b \alpha_0-4},
\end{align*}
which is strictly smaller than $\sigma^{2\beta}$ for $\alpha_0$ large. Combining \eqref{First_bound} with the above in \eqref{eq_Lemmapp}, we obtain:
\begin{align}\label{eq_Lemmapp2}
    &\E_{(X,M) \sim \Pjoint}\left[ \log \left( \frac{p_{\theta^*}^{(M)}(X^{(M)})}{p_{F,\Sigma}^{(M)}(X^{(M)})} \right)^2 \right] \nonumber \\
   &\lesssim (\sigma^{2\beta} + \varepsilon^2) \logm^2\left(\frac{\varepsilon^4}{\sigma}\right).
\end{align}
Similar arguments for $\E_{(X,M) \sim \Pjoint}\left[ \log \left( \frac{p_{\theta^*}^{(M)}(X^{(M)})}{p_{F,\Sigma}^{(M)}(X^{(M)})} \right) \right]$, using \eqref{eq_priorboundH1} in \Cref{Lem_Priorboundlemmas}, lead to the same bound up to constants. Thus, we have that
\begin{align*}
    &B_{\varepsilon, M, \sigma, \beta} \\
&\subset \left\{(F, \Sigma): \He^2(p_{\theta^*}, p_{F,\Sigma}) \lesssim \sigma^{2\beta} + \varepsilon^2 \right\}\\
& \subset \Big\{ \theta \in \Theta : \mathbb{E}_{(X,M)\sim\mathbb{P}_{\theta^*}} \left[  \log \frac{p_{\theta^*}^{(M)}\left(X^{(M)}\right)}{p_{\theta}^{(M)}\left(X^{(M)}\right)}\right] \lesssim (\sigma^{2\beta} + \varepsilon^2) \logm^2\left(\frac{\varepsilon^4}{\sigma}\right)\\
&\mathbb{E}_{(X,M)\sim\mathbb{P}_{\theta^*}} \left[ \left( \log \frac{p_{\theta^*}^{(M)}\left(X^{(M)}\right)}{p_{\theta}^{(M)}\left(X^{(M)}\right)} \right)^2 \right]  \lesssim (\sigma^{2\beta} + \varepsilon^2) \logm^2\left(\frac{\varepsilon^4}{\sigma}\right) \Big\}
\end{align*}
Combining this with the prior probability bound for $B_{\varepsilon, M, \sigma, \beta}$ in $\eqref{eq_Blowerbound}$ and the fact that $\tilde{\varepsilon} < 0.4$ (with $\tilde{\varepsilon}$ defined in Equation \ref{eq_tildeeps}), we see that we need to meet,
\begin{align}\label{eq_Cond1}
\frac{\varepsilon^4}{\sigma} < 0.4 e \max_m \|p^{(m)}_{\theta^*}\|_{\infty} 
\end{align}
\begin{align}\label{eq_Cond2}
    (\sigma^{2\beta} + \varepsilon^2) \logm^2\left(\frac{\varepsilon^4}{\sigma}\right) \lesssim \varepsilon_n^2
\end{align}
\begin{align}\label{eq_Cond3}
    (\logm(\sigma))^{d/\tau} \sigma^{-d} (\logm(\varepsilon))^{d+1} + \sigma^{-2} \lesssim n \varepsilon_n^2
\end{align}
The goal is to find $\varepsilon, \sigma$ that meet these conditions with 
$$\varepsilon_n \sim \bar{\varepsilon}_n= n^{-\beta/(2\beta + d)} (\log(n))^{(\beta d + \beta d/\tau + d + \beta)/(2\beta +d)},$$ 
as in the statement of the theorem (and as in \cite[Propositon 9.14 ii)]{Fundamentals}). We choose $\varepsilon^2 \asymp \min(\sigma^d, \sigma^{2\beta})$\footnote{On page 249 of \cite{Fundamentals} they choose $\varepsilon^4 \asymp \min(\sigma^d, \sigma^{2\beta})$, which we believe to be a typo.} and $\sigma^{2\beta + d} \asymp n^{-1} \log(n)^{d/\tau + d -1}$. This first ensures that $\varepsilon^4/\sigma$ will be much smaller than 1, ensuring \eqref{eq_Cond1} for large enough $n$. Moreover, 
\begin{align}\label{eq_secondconditioncheck}
    (\sigma^{2\beta} + \varepsilon^2) \logm^2\left(\frac{\varepsilon^4}{\sigma}\right)\leq  2\sigma^{2\beta}  (2d-1)^2 \logm^2\left(\sigma\right) \lesssim \sigma^{2\beta}\logm^2(\sigma) ,
\end{align}
and $\sigma^{2\beta}\logm^2(\sigma) \sim \bar{\varepsilon}_n^2$. Indeed as 
\[
\sigma^{2\beta}=(\sigma^{2\beta + d})^{2 \beta/(2\beta + d)} \asymp n^{-2 \beta/(2\beta + d)} (\log(n))^{(d/\tau + d -1) \cdot 2 \beta/(2\beta + d)}
\]
and
\[
\logm^2(\sigma)=\frac{\logm^2(\sigma^{2\beta + 1})}{2\beta + 1}\lesssim \logm^2(n^{-1} \log(n)^{d/\tau + d -1}) \sim \logm^{2}(n^{-1})=(\log(n))^2,
\]
we arrive at $\sigma^{2\beta}\logm^2(\sigma)  \sim n^{-2 \beta/(2\beta + d)} (\log(n))^{t_0} $, with
\begin{align*}
t_0&=\frac{2\beta(d/\tau+d-1)}{2\beta+d} + 2 
= \frac{2(\beta d/\tau + \beta d + \beta + d)}{2\beta+d}.
\end{align*}
Finally, to verify \eqref{eq_Cond3}, we first note that
\begin{align*}
     (\logm(\sigma))^{d/\tau} \sigma^{-d} (\logm(\varepsilon))^{d+1} + \sigma^{-2} &\lesssim (\logm(\sigma))^{d/\tau+d+1} \sigma^{-d} + \sigma^{-2}.
\end{align*}
Since $d \geq 2$, $\sigma^{-2}$ will be negligible compared to the first term, and we may focus on $(\logm(\sigma))^{d/\tau+d+1} \sigma^{-d}$. Then with the same tricks as before:
\begin{align*}
   (\logm(\sigma))^{d/\tau+d+1} \sigma^{-d}& \lesssim (\logm(n^{-1} \log(n)^{d/\tau + d -1}))^{d/\tau + d -1} n^{d/(2\beta + d)} (\log(n))^{-(d/\tau + d -1) \cdot d/(2\beta + d)}\\
    &\sim (\log(n))^{d/\tau + d -1} (\log(n))^{-(d/\tau + d -1) \cdot d/(2\beta + d)} n^{d/(2\beta + d)}.
\end{align*}
For the exponent of $n$, we have $n^{d/(2\beta + d)}= n n^{d/(2\beta + d)-1}=n^{d-2\beta /(2\beta + d)}$, while the power of $\log(n)$ is
\begin{align*}
d/\tau+d+1 - \frac{(d/\tau+d-1)d}{2\beta+d} 
&= \frac{(d/\tau+d+1)(2\beta+d) - (d/\tau+d-1)d}{2\beta+d}\\
&= \frac{2\beta d/\tau + 2\beta d + 2\beta + 2d}{2\beta+d}\\
&= \frac{2(\beta d/\tau + \beta d + \beta + d)}{2\beta+d},
\end{align*}
recovering $n \bar{\varepsilon}_n^2$ as desired. We thus recover the prior mass assumption with exactly the same rate as in Proposition 19.14 ii) in \cite{Fundamentals}. Combining this with Lemma 9.15 (Entropy), the proof now proceeds in exactly the same way as in Section 9.4.4 of \cite[Chapter 9]{Fundamentals}. Adding to this Lemma 9.16 for the inverse-Wishart distribution explains the factor $\kappa=2$.

\end{proof}

\MissingResponseRates*

\begin{proof}
    The only challenge here is to show that Assumption \ref{asm_prior2plus} is satisfied, in which case the theorem would be a direct application of Proposition \ref{Cor:GeneralResultH}.

    Note that parts (ii) and (iii) of Assumption \ref{asm_prior2plus} are already explicitly assumed here, so the challenge lies in showing part (i), that is:
    $\Pi\left(\mathcal{B}(\theta^*,\bar{\varepsilon}_n; \mathbb{P}_{\theta^*})\right) \geq e^{-\CI n \bar{\varepsilon}_n^2}$,
    for small enough $\bar{\varepsilon}_n^2$ and
    \[ \mathcal{B}(\theta^*,\varepsilon; \mathbb{P}_{\theta^*}) = \left\{ \theta \in \Theta : \KLMAR(P_{\theta^*}\| P_{\theta} )\leq \varepsilon^2 \, , \, \KLMARsq(P_{\theta^*}\| P_{\theta})\leq \varepsilon^2 \right\}, \]
    for \(\varepsilon > 0\).

    We know that $\KLMAR(P_{\theta^*}\| P_{\theta})\leq \KL(P_{\theta^*}\| P_{\theta})$, and we can further show that
        \begin{align*}
        \KLMARsq&(P_{\theta^*}\| P_{\theta})=\E_{(X,Y,M)\sim\mathbb{P}_{\theta^*}}\left[ \log^2 \left( \frac{p_{\theta^*}^{(M)}(X,Y^{(M)})}{p_{\theta}^{(M)}(X,Y^{(M)})} \right) \right]\\
        &=\E_{(X,Y,M)\sim\mathbb{P}_{\theta^*}}\left[ \log^2 \left( \frac{p_{\theta^*}^{(M)}(X,Y^{(M)})}{p_{\theta}^{(M)}(X,Y^{(M)})} \right) \mathbbm{1}\left(M=0\right) + \log \left( \frac{p_{\theta^*}^{(M)}(X,Y^{(M)})}{p_{\theta}^{(M)}(X,Y^{(M)})} \right) \mathbbm{1}\left(M=1\right) \right] \\
        &=\E_{(X,Y,M)\sim\mathbb{P}_{\theta^*}}\left[ \log^2 \left( \frac{p_{\theta^*}(X,Y)}{p_{\theta}(X,Y)} \right) \mathbbm{1}\left(M=0\right) + \log \left(\frac{p_X(X)}{p_X(X)} \right) \mathbbm{1}\left(M=1\right) \right] \\
        &=\E_{(X,Y,M)\sim\mathbb{P}_{\theta^*}}\left[ \log^2 \left( \frac{p_{\theta^*}(X,Y)}{p_{\theta}(X,Y)} \right) \mathbbm{1}\left(M=0\right) \right] \\
        &\leq \E_{(X,Y,M)\sim\mathbb{P}_{\theta^*}}\left[ \log^2 \left( \frac{p_{\theta^*}(X,Y)}{p_{\theta}(X,Y)} \right) \right] \\
        &=\mathbb{E}_{P_{\theta^*}} \left[ \left( \log \frac{p_{\theta^*}}{p_{\theta}} \right)^2 \right] \, 
    \end{align*}
    by denoting $p_\theta$ the joint modeled density over $(X,Y)$ and $p_X$ the (fixed and known) marginal density of $X$.

    Hence, we have $\mathcal{B}(\theta^*,\bar{\varepsilon}_n)\subset \mathcal{B}(\theta^*,\bar{\varepsilon}_n; \mathbb{P}_{\theta^*})$, and then $\pi(\mathcal{B}(\theta^*,\bar{\varepsilon}_n; \mathbb{P}_{\theta^*}))\geq \pi(\mathcal{B}(\theta^*,\bar{\varepsilon}_n))$, which ends the proof.
\end{proof}

\MissingCovariatesRates*

\begin{proof}
    As in the proof above, the main challenge here lies in showing that
    $\Pi\left(\mathcal{B}(\theta^*,\bar{\varepsilon}_n; \mathbb{P}_{\theta^*})\right) \geq e^{-\CI n \bar{\varepsilon}_n^2}$,
    for small enough $\bar{\varepsilon}_n^2$ and
    \[ \mathcal{B}(\theta^*,\varepsilon; \mathbb{P}_{\theta^*}) = \left\{ \theta \in \Theta : \KLMAR(P_{\theta^*}\| P_{\theta} )\leq \varepsilon^2 \, , \, \KLMARsq(P_{\theta^*}\| P_{\theta})\leq \varepsilon^2 \right\}, \]
    for \(\varepsilon > 0\).

    To do so, we simply have to show that both $\KLMAR(P_{\theta^*}\| P_{\theta})$ and $\KLMARsq(P_{\theta^*}\| P_{\theta})$ are upper bounded by $\|f_\theta-f_{\theta^\star}\|_\infty$ (up to a factor), which will end the proof.

    We start by observing that the ratio of two $\sigma$-variance-Gaussian densities $\phi_\sigma$ satisfies for any \(x,y,\theta\):
    \[ 
    \frac{\phi_\sigma(y-f_\theta(x))}{\phi_\sigma(y-f_{\theta^\star}(x))} = \exp\left( \frac{(f_\theta-f_{\theta^\star})(x)}{\sigma^2} \left(y-\frac{f_\theta(x)+f_{\theta^\star}(x)}{2}\right) \right) \leq \exp\left( \frac{\|f_\theta-f_{\theta^\star}\|_\infty}{\sigma^2} \left(|y|+B\right) \right) \, .
    \]
    Consequently, we have
    \[ 
    \phi_\sigma(y-f_\theta(x)) \leq \exp\left( \frac{\|f_\theta-f_{\theta^\star}\|_\infty}{\sigma^2} \left(|y|+B\right) \right) \, \phi_\sigma(y-f_{\theta^\star}(x)) \, .
    \]    
    Then, the observed-data likelihood can be controlled as follows:
    \begin{align*}
    p_\theta^{(m)}\left(x^{(m)},y\right) &= \int_{x^{(-m)}} p_\theta(x,y) \,\mathrm{d}x^{(-m)} \\
    &= \int_{x^{(-m)}} p_X(x) \, \phi_\sigma(y-f_\theta(x)) \,\mathrm{d}x^{(-m)} \\
    &\leq \int_{x^{(-m)}} p_X(x)\, \exp\left( \frac{\|f_\theta-f_{\theta^\star}\|_\infty}{\sigma^2} \left(|y|+B\right) \right) \, \phi_\sigma(y-f_{\theta^\star}(x)) \,\mathrm{d}x^{(-m)} \\
    &= \exp\left( \frac{\|f_\theta-f_{\theta^\star}\|_\infty}{\sigma^2} \left(|y|+B\right) \right)\int_{x^{(-m)}} p_X(x) \, \phi_\sigma(y-f_{\theta^\star}(x)) \,\mathrm{d}x^{(-m)} \\
    &= \exp\left( \frac{\|f_\theta-f_{\theta^\star}\|_\infty}{\sigma^2} \left(|y|+B\right) \right) p_{\theta^\star}^{(m)}\left(x^{(m)},y\right) \, .
    \end{align*}
    Hence, we have the following bound independent of $x$ and $m$,
    $$
    \log\left(\frac{p_{\theta^\star}^{(m)}\left(x^{(m)},y\right)}{p_\theta^{(m)}\left(x^{(m)},y\right)}\right) \leq \frac{\|f_\theta-f_{\theta^\star}\|_\infty}{\sigma^2} \left(|y|+B\right) \, ,
    $$
    and we have both
    \begin{align*}
        \KLMAR(P_{\theta^*}\| P_{\theta})&=\E_{(X,Y,M)\sim\mathbb{P}_{\theta^*}}\left[ \log\left( \frac{p_{\theta^*}^{(M)}(X^{(M)},Y)}{p_{\theta}^{(M)}(X^{(M)},Y)} \right) \right]\\
        &\leq\E_{(X,Y,M)\sim\mathbb{P}_{\theta^*}}\left[ \frac{\|f_\theta-f_{\theta^\star}\|_\infty}{\sigma^2} \left(|Y|+B\right) \right] \\
        &=\frac{\|f_\theta-f_{\theta^\star}\|_\infty}{\sigma^2} \left(\E\left[|Y|\right]+B\right) \\
        &\leq\frac{2B+\sigma}{\sigma^2} \cdot \|f_\theta-f_{\theta^\star}\|_\infty \,
    \end{align*}
    and
    \begin{align*}
        \KLMARsq(P_{\theta^*}\| P_{\theta})&=\E_{(X,Y,M)\sim\mathbb{P}_{\theta^*}}\left[ \log^2\left( \frac{p_{\theta^*}^{(M)}(X^{(M)},Y)}{p_{\theta}^{(M)}(X^{(M)},Y)} \right) \right]\\
        &\leq\E_{(X,Y,M)\sim\mathbb{P}_{\theta^*}}\left[ \frac{\|f_\theta-f_{\theta^\star}\|^2_\infty}{\sigma^4} \left(2|Y|^2+2B^2\right) \right] \\
        &=\frac{4B^2+2\sigma^2}{\sigma^4} \cdot \|f_\theta-f_{\theta^\star}\|^2_\infty \\
        &\leq\frac{4B^2+2\sigma^2}{\sigma^4} \cdot \|f_\theta-f_{\theta^\star}\|_\infty \, ,
    \end{align*}
    the last inequality holding as soon as $\|f_\theta-f_{\theta^\star}\|_\infty\leq1$, which holds under the considered neighborhood sizes.
\end{proof}

\section{Additional Results}

We first recall that
\[ \mathcal{B}(\theta^*,\varepsilon; \mathbb{P}_{\theta^*}) = \left\{ \theta \in \Theta : \mathbb{E}_{(X,M)\sim\mathbb{P}_{\theta^*}} \left[  \log \frac{p_{\theta^*}^{(M)}\left(X^{(M)}\right)}{p_{\theta}^{(M)}\left(X^{(M)}\right)}\right] \leq \varepsilon^2 \, , \, \mathbb{E}_{(X,M)\sim\mathbb{P}_{\theta^*}} \left[  \log \left(\frac{p_{\theta^*}^{(M)}\left(X^{(M)}\right)}{p_{\theta}^{(M)}\left(X^{(M)}\right)} \right)^2 \right]  \leq \varepsilon^2 \right\}, \]
for \(\varepsilon > 0\).

\begin{lemma}\label{Omegaanbound}
    Let for arbitrary $C, r > 0$,
    \begin{align*}
    \Omega_n(r) = \left\{ \int \prod_{i=1}^n \frac{p_{\theta}^{(M_i)}\left(X_i^{(M_i)}\right)}{p_{\theta^*}^{(M_i)}\left(X_i^{(M_i)}\right)} d\Pi(\theta) \geq \Pi(\mathcal{B}(\theta^*,r; \mathbb{P}_{\theta^*})) e^{-n r^2 (1+C)} \right\} .
\end{align*}
Then
\begin{align}
    \mathbb{P}_{\theta^*}(\Omega_n(r)^c) \leq  \frac{1}{C^2 n r^2}.
\end{align}
\end{lemma}

\begin{proof}
We largely follow the proof in Lemma 8.10 \cite{Fundamentals}. First, taking $B=\mathcal{B}(\theta^*,r; \mathbb{P}_{\theta^*})$, we notice that the inequality in $\Omega_n(r)$ is always true if $\Pi(B)=0$. Thus, we may assume $\Pi(B) > 0$. Similarly, since
    \begin{align*}
     \int \prod_{i=1}^n \frac{p_{\theta}^{(M_i)}\left(X_i^{(M_i)}\right)}{p_{\theta^*}^{(M_i)}\left(X_i^{(M_i)}\right)} d\Pi(\theta) \geq \int_B \prod_{i=1}^n \frac{p_{\theta}^{(M_i)}\left(X_i^{(M_i)}\right)}{p_{\theta^*}^{(M_i)}\left(X_i^{(M_i)}\right)} d\Pi(\theta)
    \end{align*}
and since we can divide both sides in the inequality by $\Pi(B)$, we can without loss of generality assume that $\Pi(B)=1$. We now apply Jensen's inequality:
\begin{align*}
    \log \left(  \int \prod_{i=1}^n \frac{p_{\theta}^{(M_i)}\left(X_i^{(M_i)}\right)}{p_{\theta^*}^{(M_i)}\left(X_i^{(M_i)}\right)} d\Pi(\theta)\right) \geq \sum_{i=1}^n \int   \log \left(\frac{p_{\theta}^{(M_i)}\left(X_i^{(M_i)}\right)}{p_{\theta^*}^{(M_i)}\left(X_i^{(M_i)}\right)}\right) d\Pi(\theta):=Z(\data).
\end{align*}
Then, using a change of integral, it holds that,
\[
\E_{\data \sim \Pjointn}[Z(\data)]=n \E_{(X,M)\sim \Pjoint}\left[  \log \left(\frac{p_{\theta}^{(M_i)}\left(X_i^{(M_i)}\right)}{p_{\theta^*}^{(M_i)}\left(X_i^{(M_i)}\right)}\right) \right]=-n\KLMAR(P_{\theta^*} \| P_{\theta}) > -n r^2.
\]
Thus,
\begin{align}\label{eq_Zprobbound}
    \mathbb{P}_{\theta^*}(Z(\data) < \log(e^{-n r^2 (1+C)}))&= \mathbb{P}_{\theta^*}(Z(\data)  < -n r^2 (1+C))\nonumber\\
    &\leq \mathbb{P}_{\theta^*}(Z(\data) < \E_{\data \sim \Pjointn}[Z(\data)]-n r^2C)\nonumber\\
    &\leq \mathbb{P}_{\theta^*}(|Z(\data)-\E_{\data \sim \Pjointn}[Z(\data)]| > n r^2C) \nonumber \\
     &\leq \E_{\data \sim \Pjointn} \left[ \frac{|Z(\data)-\E_{\data \sim \Pjointn}[Z(\data)]|^2}{n^2 r^4 C^2} \right]
\end{align}
where we used Markov's inequality in the last step. Now using Jensen
\begin{align}\label{eq_Zexpbound}
    &\E_{\data \sim \Pjointn} \left[ |Z(\data)-\E_{\data \sim \Pjointn}[Z(\data)]|^2\right] \nonumber\\
    &\leq \sum_{i=1}^{n} \E_{\data \sim \Pjointn} \left[\int \left(\log \left(\frac{p_{\theta}^{(M_i)}\left(X_i^{(M_i)}\right)}{p_{\theta^*}^{(M_i)}\left(X_i^{(M_i)}\right)}\right) - \KLMAR(P_{\theta^*} \| P_{\theta})\right)^2 d\Pi(\theta) \right]\nonumber\\
    &=n\int \E_{(X,M) \sim \Pjoint} \left[\left(\log \left(\frac{p_{\theta}^{(M)}\left(X^{(M)}\right)}{p_{\theta^*}^{(M)}\left(X^{(M)}\right)}\right) - \KLMAR(P_{\theta^*} \| P_{\theta})\right)^2 \right]d\Pi(\theta) \nonumber\\
    & \leq n\int \E_{(X,M) \sim \Pjoint} \left[\left(\log \left(\frac{p_{\theta}^{(M)}\left(X^{(M)}\right)}{p_{\theta^*}^{(M)}\left(X^{(M)}\right)}\right)\right)^2 \right]d\Pi(\theta)\nonumber\\
        & \leq n r^2. 
\end{align}
Combining \eqref{eq_Zprobbound} and \eqref{eq_Zexpbound} with the fact that, by the definition of $Z(\data)$,
\begin{align*}
    \mathbb{P}_{\theta^*}(\Omega_n(r)^c) \leq  \mathbb{P}_{\theta^*}(Z(\data) < \log(e^{-n r^2 (1+C)})),
\end{align*}
gives the result.
\end{proof}

Recall the definition of $\HeMAR$ in the main text:
\begin{align*}
    \HeMAR^2(p_{\theta_1}, p_{\theta_2})=\sum_{m} \int |\sqrt{p_{\theta_1}^{(m)}}(x^{(m)}) - \sqrt{p_{\theta_2}^{(m)}}(x^{(m)})|^2 \Prob(M=m \mid x^{(m)}) \mathrm{d}x^{(m)}. 
\end{align*}

We also adapt the Hellinger affinity, $\Af(p_{\theta_1}; p_{\theta_2}) = \int \sqrt{p_{\theta_1}}(x)\sqrt{p_{\theta_2}}(x)\,\mathrm{d}x $, as
\begin{align*}
    \AfMAR(p_{\theta_1}; p_{\theta_2}) = \sum_{m} \int \sqrt{p^{(m)}_{\theta_1}}(x^{(m)})\sqrt{p^{(m)}_{\theta_2}}(x^{(m)}) \Prob(M=m \mid x^{(m)})\,\mathrm{d}x^{(m)}
\end{align*}

We then have that:
\begin{align*}
     &\HeMAR^2(p_{\theta_1}, p_{\theta_2})\\
     &=\sum_{m} \int |\sqrt{p_{\theta_1}^{(m)}}(x^{(m)}) - \sqrt{p_{\theta_2}^{(m)}}(x^{(m)})|^2 \Prob(M=m \mid x^{(m)}) \mathrm{d}x^{(m)}\\
     &=\sum_{m} \Big( \int p_{\theta_1}^{(m)}(x^{(m)}) \Prob(M=m \mid x^{(m)}) \mathrm{d}x^{(m)} + \int p_{\theta_2}^{(m)}(x^{(m)})\Prob(M=m \mid x^{(m)}) \mathrm{d}x^{(m)}\\
     &- 2\int \sqrt{p_{\theta_1}^{(m)}}(x^{(m)})\sqrt{p_{\theta_2}^{(m)}}(x^{(m)}) \Prob(M=m \mid x^{(m)}) \mathrm{d}x^{(m)}\Big)\\
        &=  2- 2 \AfMAR(p_{\theta_1}; p_{\theta_2}).
\end{align*}

In the following, we recall that $\Theta=\mathcal{P}$ is dominated by the Lebesgue measure by assumption.

\begin{thm}\label{amazintheorem}
Assume that Assumptions \ref{asm_true_MDM} and \ref{asm_no_empty_MDM} hold. Moreover, assume that there exists $\theta_{\star} \in \Theta$ such that $x \mapsto p_{\theta_{\star}}(x)$ is continuous and $\supp(p_{\theta}) \subset \supp(p_{\theta_{\star}})$ for all $\theta \in \Theta$. Then given any $p_{\theta_1}$ and $p_{\theta_2}$, there exist probability densities $\bar{p}_{\theta_1}$ and $\bar{p}_{\theta_2}$ such that for any probability density $p_{\theta_3}$,
\begin{align}
\E_{(X,M) \sim \Prob_{\theta_3}}\left[\sqrt{\frac{\bar{p}^{(M)}_{\theta_2}(X^{(M)})}{\bar{p}^{(M)}_{\theta_1}(X^{(M)})}}\right]\leq 1 - \frac{1}{6}\HeMAR^2(p_{\theta_1}, p_{\theta_2}) + \HeMAR^2(p_{\theta_2}, p_{\theta_3}),
\end{align}
and 
\begin{align}
   \E_{(X,M) \sim \Prob_{\theta_3}}\left[\sqrt{\frac{\bar{p}^{(M)}_{\theta_1}(X^{(M)})}{\bar{p}^{(M)}_{\theta_2}(X^{(M)})}}\right]  \leq 1 - \frac{1}{6}\HeMAR^2(p_{\theta_1}, p_{\theta_2}) + \HeMAR^2(p_{\theta_2}, p_{\theta_3})
\end{align}
\end{thm}

\begin{proof}

For arbitrary $p_{\theta_1}, p_{\theta_2}$ we construct a joint densities $q_{\theta_1},q_{\theta_2} $ of $(X,M)$ on $\mathcal{X} \times \{0,1\}^d$: For $j \in \{1,2\}$, and $x \in \mathcal{X}$
\begin{align}
q_{\theta_j}(x,m)=p_{\theta_j}^{(m)}(x^{(m)}) p_{\theta_{\star}}(x^{(-m)} \mid x^{(m)}) \Prob(M=m \mid x).    
\end{align}
Due to the fact that $\supp(p_{\theta_j}) \subset \supp(p_{\theta_{\star}}) $, and that $p_{\theta_{\star}}$ is continuous $p_{\theta_{\star}}(x^{(-m)} \mid x^{(m)})$ is well-defined in the support of $p_{\theta_j}^{(m)}$ and $q_{\theta_j}(x,m)$ is well-defined for almost all $(x,m)$. Moreover, it is a valid density with respect to the product of the Lebesgue and counting measure, as thanks to MAR, 
\begin{align*}
   \sum_m \int  q_{\theta_j}(x,m) \mathrm{d}x
   &=\sum_m \int p_{\theta_j}^{(m)}(x^{(m)}) \int  \Prob(M=m \mid x) p_{\theta_{\star}}(x^{(-m)} \mid x^{(m)}) d x^{(-m)}  \mathrm{d}x^{(m)}\\
      &=\sum_m \int p_{\theta_j}^{(m)}(x^{(m)}) \int  \Prob(M=m \mid x) p_{\theta_{j}}(x^{(-m)} \mid x^{(m)}) d x^{(-m)}  \mathrm{d}x^{(m)}\\
      &=\sum_m \int p_{\theta_j}(x)  \Prob(M=m \mid x) d x\\
      &=1.
\end{align*}
As such, we can use Lemma D.2. of \cite{Fundamentals} to obtain $ \bar{q}_{\theta_j}$ with 
\begin{align}\label{eq_partoflincomb}
    \sqrt{\bar{q}_{\theta_j}} \in \{a \sqrt{q_{\theta_1}} + b \sqrt{q_{\theta_2}} : (a,b) \in \R^2, \|a \sqrt{q_{\theta_1}} + b \sqrt{q_{\theta_2}} \|_{L_2(\nu)}=1 \},
\end{align}
and
\begin{align}\label{eq_naivepplication}
   \E_{(X,M) \sim \Prob_{\theta_3}}\left[\sqrt{\frac{\bar{q}_{\theta_2}(X,M)}{\bar{q}_{\theta_1}(X,M)}}\right]\leq 1 - \frac{1}{6}\He^2(q_{\theta_1}, q_{\theta_2}) + \He^2(q_{\theta_2}, q_{\theta_3}).
\end{align}
Analyzing the right side of \eqref{eq_naivepplication}, we notice that under MAR,
\begin{align*}
    \He^2(q_{\theta_1}, q_{\theta_2})&=\sum_{m} \int |\sqrt{p_{\theta_1}^{(m)}}(x^{(m)}) - \sqrt{p_{\theta_2}^{(m)}}(x^{(m)})|^2 p_{\theta_{\star}}(x^{-(m)} \mid x^{(m)}) \Prob(M=m \mid x)\mathrm{d}x\\
    &=\sum_{m} \int |\sqrt{p_{\theta_1}^{(m)}}(x^{(m)}) - \sqrt{p_{\theta_2}^{(m)}}(x^{(m)})|^2 \Prob(M=m \mid x^{(m)}) \int p_{\theta_{\star}}(x^{-(m)} \mid x^{(m)}) \mathrm{d}x^{(-m)} \mathrm{d}x^{(m)}\\
    &=\sum_{m} \int |\sqrt{p_{\theta_1}^{(m)}}(x^{(m)}) - \sqrt{p_{\theta_2}^{(m)}}(x^{(m)})|^2 \Prob(M=m \mid x^{(m)})  \mathrm{d}x^{(m)}\\
    &=\HeMAR^2(p_{\theta_1}, p_{\theta_2}),
\end{align*}
and analogously, $\He^2(q_{\theta_2}, q_{\theta_3})=\HeMAR^2(p_{\theta_2}, p_{\theta_3})$.

Considering the left hand side of \eqref{eq_naivepplication}, we first note that \eqref{eq_partoflincomb} implies that $\sqrt{\bar{q}_{\theta_j}}$ has the form
\begin{align*}
    \sqrt{\bar{q}_{\theta_j}(x,m)} &=a \sqrt{q_{\theta_1}(x,m)} + b \sqrt{q_{\theta_2}(x,m)}\\
    &=(a \sqrt{p_{\theta_1}^{(m)}(x^{(m)})} + b \sqrt{p_{\theta_2}^{(m)}(x^{(m)})}) \sqrt{p_{\theta_{\star}}(x^{(-m)} \mid x^{(m)}) \Prob(M=m \mid x)},
\end{align*}
for all $m$ and almost all $x$. Thus, defining pointwise a.s, $\sqrt{\bar{p}_{\theta_j}^{(m)}(x^{(m)})}=a \sqrt{p_{\theta_1}^{(m)}(x^{(m)})} + b \sqrt{p_{\theta_2}^{(m)}(x^{(m)})}$, we see that under MAR,
\begin{align*}
      &\E_{(X,M) \sim \Prob_{\theta_3}}\left[\sqrt{\frac{\bar{q}_{\theta_2}(X,M)}{\bar{q}_{\theta_1}(X,M)}}\right]\\
      &= \sum_{m} \int \sqrt{\frac{\bar{q}_{\theta_2}(x,m)}{\bar{q}_{\theta_1}(x,m)}} p_{\theta_3}(x) \Prob(M=m \mid x) \mathrm{d}x  \\
      &=\sum_{m} \int \sqrt{\frac{\bar{p}^{(m)}_{\theta_2}(x^{(m)})}{\bar{p}^{(m)}_{\theta_1}(x^{(m)})}} p_{\theta_3}(x) \Prob(M=m \mid x) \mathrm{d}x \\
        &=\E_{(X,M) \sim \Prob_{\theta_3}}\left[\sqrt{\frac{\bar{p}^{(M)}_{\theta_2}(X^{(M)})}{\bar{p}^{(M)}_{\theta_1}(X^{(M)})}}\right].
\end{align*}
This shows the result. The second inequality, follows with analogous arguments.
\end{proof}

\begin{cor}[Basic Hellinger testing under MAR Missingness]\label{cor_basictest}
Assume that Assumptions \ref{asm_true_MDM} -- \ref{asm_positive_MDM} hold. Moreover, assume that there exists $\theta_{\star} \in \Theta$ such that $x \mapsto p_{\theta_{\star}}(x)$ is continuous and $\supp(p_{\theta}) \subset \supp(p_{\theta_{\star}})$ for all $\theta \in \Theta$. Then given any $n$, $\varepsilon > 0$ and $p_{\theta_1}$, with $\He(p_{\theta^*},p_{\theta_1}) > \varepsilon$, there exists a test $\psi_n=\psi_n(\data)$, such that
\begin{align}
    \E_{\data\sim \Pjointn} [\psi_n(\data)] \leq e^{-n \frac{\delta}{8}\varepsilon^2}
\end{align}
and
\begin{align}
  \sup_{\theta: \He(p_{\theta^*},p_{\theta}) < \varepsilon \sqrt{\delta}/\sqrt{24}  }  \E_{\data\sim \Pthetajointn} (1-\psi_n(\data)) \leq e^{-n \frac{\delta}{8}\varepsilon^2}.
\end{align}
\end{cor}

\begin{proof}
    Let $\bar{p}_{\theta^*}, \bar{p}_{\theta_1}$ be the densities attached to $p_{\theta^*}, p_{\theta_1}$ as in Theorem \ref{amazintheorem}. Then, we define the test $\psi_n$ as:
    \begin{align}
        \psi_n(\data)=\mathbf{1}\left\{ \prod_{i=1}^{n}\sqrt{\frac{\bar{p}^{(M_i)}_{\theta_1}(X_i^{(M_i)})}{\bar{p}^{(M_i)}_{\theta^*}(X_i^{(M_i)})}} > 1\right\}.
    \end{align}
    Then using Markov's inequality, the i.i.d. sampling of $(X_i^{(M_i)}, M_i)$ Theorem \ref{amazintheorem} and \Cref{Hellingerbounding},
    \begin{align*} 
        \E_{\data\sim \Pjointn} [\psi_n(\data)]\leq \left( 1 - \frac{1}{6}\HeMAR^2(p_{\theta^*}, p_{\theta_1}) \right)^n
        \leq \left( 1 - \frac{\delta}{6}\He^2(p_{\theta^*}, p_{\theta_1}) \right)^n
        \leq e^{-n \frac{\delta}{6}\He^2(p_{\theta^*}, p_{\theta_1})}\leq e^{-n \frac{\delta}{8}\varepsilon^2}
    \end{align*}
 Similarly for any $\theta$ such that $\He(p_{\theta^*},p_{\theta}) < \varepsilon/5 $   
 \begin{align*}
     \E_{\data\sim \Pthetajointn} (1-\psi_n(\data)) &\leq \left( \E_{(X,M) \sim \Prob_{\theta}}\left[\sqrt{\frac{\bar{p}^{(M)}_{\theta^*}(X^{(M)})}{\bar{p}^{(M)}_{\theta_1}(X^{(M)})}}\right]\right)^n\\
     & \leq \left( 1 - \frac{1}{6}\HeMAR^2(p_{\theta^*}, p_{\theta_1}) + \HeMAR^2(p_{\theta_1}, p_{\theta})\right)^n\\
     & \leq \left( 1 - \frac{\delta}{6}\He^2(p_{\theta^*}, p_{\theta_1}) + \He^2(p_{\theta_1}, p_{\theta})\right)^n\\
     & \leq e^{-n \frac{\delta}{6}\He^2(p_{\theta^*}, p_{\theta_1})  + n\He^2(p_{\theta_1}, p_{\theta})}
 \end{align*}
Choosing $\He^2(p_{\theta_1}, p_{\theta}) <  \delta \varepsilon^2/24$, gives 
\begin{align*}
    -\frac{\delta}{6}\He^2(p_{\theta^*}, p_{\theta_1})  + \He^2(p_{\theta_1}, p_{\theta}) <-\frac{\delta \varepsilon^2}{6}  + \frac{\varepsilon^2 \delta}{24} = \frac{\delta \varepsilon^2}{8},
\end{align*}
 leading to the result.
\end{proof}

\begin{lemma}\label{Lem_Priorboundlemmas}
Assume that there exists $\theta_{\star} \in \Theta$ such that $x \mapsto p_{\theta_{\star}}(x)$ is continuous and $\supp(p_{\theta}) \subset \supp(p_{\theta_{\star}})$ for all $\theta \in \Theta$. Then for every $p_{\theta_1}$ and $p_{\theta_2}$ and any $0 < \epsilon < 0.4$,
\begin{align}\label{eq_priorboundH1}
   \E_{(X,M)\sim \Prob_{\theta_1}} \left[ \log\left( \frac{p_{\theta_1}(X^{(M)})}{p_{\theta_2}(X^{(M)})} \right) \right] &\leq \HeMAR^2(p_{\theta_1},p_{\theta_2})(1 + 2\logm\epsilon) \nonumber \\ 
    &+ 2\E_{(X,M)\sim \Prob_{\theta_1}}\!\left[\log\left(\frac{p_{\theta_1}(X^{(M)})}{p_{\theta_2}(X^{(M)})}\right)\mathbf{1}\left \{\frac{p_{\theta_2}(X^{(M)})}{p_{\theta_1}(X^{(M)})} \leq \epsilon \right \}\right].
\end{align}
and
\begin{align}\label{eq_priorboundH2}
\E_{(X,M)\sim \Prob_{\theta_1}} \left[ \log\left( \frac{p_{\theta_1}(X^{(M)})}{p_{\theta_2}(X^{(M)})} \right)^2 \right] &\leq \HeMAR^2(p_{\theta_1},p_{\theta_2})(12 + 2\logm^2\epsilon)\\ 
    &+ 8\E_{(X,M)\sim \Prob_{\theta_1}}\!\left[\log\left(\frac{p_{\theta_1}^{(M)}(X^{(M)})}{p_{\theta_2}^{(M)}(X^{(M)})}\right)^2\mathbf{1}\left \{\frac{p_{\theta_2}^{(M)}(X^{(M)})}{p_{\theta_1}^{(M)}(X^{(M)})} \leq \epsilon \right \}\right].
\end{align}
\end{lemma}

\begin{proof}
Defining again for $j \in \{1,2 \}$,
\begin{align*}
q_{\theta_j}(x,m)=p_{\theta_j}^{(m)}(x^{(m)}) p_{\theta_{\star}}(x^{(-m)} \mid x^{(m)}) \Prob(M=m \mid x),    
\end{align*}
with distribution $Q_{\theta_j}$, we can use Lemma B.2. of \cite[Appendix B]{Fundamentals} to get that for $0 < \varepsilon < 0.4$,
\begin{align*}
\E_{(X,M)\sim Q_{\theta_1}} \left[ \log\left( \frac{q_{\theta_1}(X, M)}{q_{\theta_2}(X, M)} \right) \right] &\leq \He^2(q_{\theta_1},q_{\theta_2})(2 + 2\logm\epsilon)\\ 
    &+ 2\E_{(X,M)\sim Q_{\theta_1}}\!\left[\log\left(\frac{q_{\theta_1}(X, M)}{q_{\theta_2}(X, M)}\right)\mathbf{1}\left \{\frac{q_{\theta_2}(X, M)}{q_{\theta_1}(X, M)} \leq \epsilon \right \}\right].
\end{align*}
and
\begin{align*}
\E_{(X,M)\sim Q_{\theta_1}} \left[ \log\left( \frac{q_{\theta_1}(X, M)}{q_{\theta_2}(X, M)} \right)^2 \right] &\leq \He^2(q_{\theta_1},q_{\theta_2})(12 + 2\log_-^2\epsilon)\\ 
    &+ 8\E_{(X,M)\sim Q_{\theta_1}}\!\left[\log\left(\frac{q_{\theta_1}(X, M)}{q_{\theta_2}(X, M)}\right)^2\mathbf{1}\left \{\frac{q_{\theta_2}(X, M)}{q_{\theta_1}(X, M)} \leq \epsilon \right \}\right].
\end{align*}
It was already shown in \Cref{amazintheorem} that $\He^2(q_{\theta_1},q_{\theta_2})=\HeMAR^2(p_{\theta_1},p_{\theta_2})$ under MAR. Moreover, we see that
\begin{align*}
    &\E_{(X,M)\sim Q_{\theta_1}} \left[ \log\left( \frac{q_{\theta_1}(X, M)}{q_{\theta_2}(X, M)} \right)^2 \right]\\
    &=\sum_{m} \int \log\left( \frac{q_{\theta_1}(x, m)}{q_{\theta_2}(x, m)} \right)^2 q_{\theta_1}(x, m) \mathrm{d}x \\
     &=\sum_{m} \int \log\left( \frac{p_{\theta_1}^{(m)}(x^{(m)})}{p_{\theta_2}^{(m)}(x^{(m)})} \right)^2 p_{\theta_1}^{(m)}(x^{(m)}) \int p_{\theta_{\star}}(x^{(-m)} \mid x^{(m)}) \Prob(M=m \mid x) \mathrm{d}x^{(-m)} \mathrm{d}x^{(m)}\\
     &=\sum_{m} \int \log\left( \frac{p_{\theta_1}^{(m)}(x^{(m)})}{p_{\theta_2}^{(m)}(x^{(m)})} \right)^2 p_{\theta_1}^{(m)}(x^{(m)}) \int p_{\theta_1}(x^{(-m)} \mid x^{(m)}) \Prob(M=m \mid x) \mathrm{d}x^{(-m)} \mathrm{d}x^{(m)}\\
     &=\sum_{m} \int \log\left( \frac{p_{\theta_1}^{(m)}(x^{(m)})}{p_{\theta_2}^{(m)}(x^{(m)})} \right)^2 p_{\theta_1}(x)   \Prob(M=m \mid x) \mathrm{d}x\\
     &=\E_{(X,M)\sim \Prob_{\theta_1}} \left[ \log\left( \frac{p_{\theta_1}(X^{(M)})}{p_{\theta_2}(X^{(M)})} \right)^2 \right],
\end{align*}
where we again used the MAR condition. With exactly the same reasoning, we have
\begin{align*}
    \E_{(X,M)\sim Q_{\theta_1}}\!\left[\log\left(\frac{q_{\theta_1}(X, M)}{q_{\theta_2}(X, M)}\right)^2\mathbf{1}\left \{\frac{q_{\theta_2}(X, M)}{q_{\theta_1}(X, M)} \leq \epsilon \right \}\right]=\E_{(X,M)\sim \Prob_{\theta_1}}\!\left[\log\left(\frac{q_{\theta_1}(X, M)}{q_{\theta_2}(X, M)}\right)^2\mathbf{1}\left \{\frac{q_{\theta_2}(X, M)}{q_{\theta_1}(X, M)} \leq \epsilon \right \}\right].
\end{align*}
This shows \eqref{eq_priorboundH2}, while \eqref{eq_priorboundH1} follows analogously.
\end{proof}

\begin{lemma}\label{lem_Gaussianbound}
    Let for $y=(y_1,y_2)$, $y_1 \in \R^{d_1}$, $y_2 \in \R^{d_2}$, $d_1 + d_2=d$, 
    \[
    p(y)=\int \phi(y-z; \Sigma) dF(z),
    \]
    for an arbitrary cdf $F: \R^d \to [0,1]$. We partition $\Sigma$ according to the dimension of $y_1$, $y_2$:
\[
    \Sigma = \begin{pmatrix} \Sigma_{11} & \Sigma_{12} \\ \Sigma_{21} & \Sigma_{22} \end{pmatrix}.
\]
Then for $p^{(2)}(y_2)=\int \phi(y_2-z_2;\Sigma_{22} ) dF(z_2)$ and $\Sigma_{11\cdot 2} = \Sigma_{11} - \Sigma_{12}\Sigma_{22}^{-1}\Sigma_{21}$,
\begin{align}\label{eq_conditionalexpression}
    p(y_1 \mid y_2)= \int \phi \left(y_1 - z_1 - \Sigma_{12}\Sigma_{22}^{-1}(y_2 - z_2), \Sigma_{11\cdot 2}\right) \frac{\phi(y_2 - z_2, \Sigma_{22})}{p^{(2)}(y_2)} dF(z),
\end{align}
and, with $\lambda_1(\Sigma) \leq \lambda_2(\Sigma) \leq \ldots, \lambda_d(\Sigma)$ the ordered eigenvalues of $\Sigma$,
\begin{align}\label{eq_conditionalbound}
    \sup_{y_1, y_2}  p(y_1 \mid y_2)\leq  \frac{1}{\left(2\pi\lambda_{1}(\Sigma)\right)^{d_1/2}}.
\end{align}

\end{lemma}

\begin{proof}

By Fubini and the fact that the marginal of a Gaussian is again Gaussian, the marginal $p^{(2)}(y_2)$ is given as
\begin{align*}
    p^{(2)}(y_2)=\int \int \phi(y-z; \Sigma) dF(z) dy_2 =\int \phi(y_2-z_2;\Sigma_{22} ) dF(z_2),
\end{align*}
as in the Theorem statement. Next, using the well-known form of the Gaussian conditional distribution, we can write
\begin{align*}
    \int \phi(y-z; \Sigma) dF(z)=\int \phi \left(y_1 - z_1 - \Sigma_{12}\Sigma_{22}^{-1}(y_2 - z_2), \Sigma_{11\cdot 2}\right) \phi(y_2 - z_2, \Sigma_{22})dF(z),
\end{align*}
whereby we split up $\phi(y-z; \Sigma)=q(y_1,y_2 \mid z)$ into $q(y_1 \mid y_2, z)$ (the conditional Gaussian in the integral) and $q(y_2\mid z)$ (the marginal Gaussian in the integral). Thus, the conditional distribution becomes
\begin{align*}
    p(y_1 \mid y_2) = \frac{p(y_1, y_2)}{p^{(2)}(y_2)}=\int \phi \left(y_1 - z_1 - \Sigma_{12}\Sigma_{22}^{-1}(y_2 - z_2), \Sigma_{11\cdot 2}\right) \frac{\phi(y_2 - z_2, \Sigma_{22})}{p^{(2)}(y_2)} dF(z),
\end{align*}
as claimed.

We now bound $\phi \left(y_1 - z_1 - \Sigma_{12}\Sigma_{22}^{-1}(y_2 - z_2), \Sigma_{11\cdot 2}\right)$. First, the largest value of a Gaussian distribution with covariance matrix $\Sigma_{11\cdot 2}$ (independent of its mean) is bounded as
\begin{align*}
  \sup_{y \in \R^{d_1}}\phi \left(y_1 - z_1 - \Sigma_{12}\Sigma_{22}^{-1}(y_2 - z_2), \Sigma_{11\cdot 2}\right) \leq \frac{1}{(2\pi)^{d_1/2}\det(\Sigma_{11\cdot 2})^{1/2}}.
\end{align*}
Since this expression is independent of the mean $z_1, y_2$, it in fact holds uniformly over all $y_1$, $y_2$ and $z_2$. Moreover, using the Schur property of \cite[Theorem 5]{Schurproperty}, we have
\[
\lambda_1(\Sigma) \leq \lambda_i(\Sigma_{11\cdot 2}) \leq \lambda_d(\Sigma).
\]
Since $\det(\Sigma_{11\cdot 2})=\prod_{j=}^{d_1} \lambda_j(\Sigma_{11\cdot 2})$, we may further bound 
\begin{align*}
\frac{1}{(2\pi)^{d_1/2}\det(\Sigma_{11\cdot 2})^{1/2}} \leq \frac{1}{(2\pi)^{d_1/2}\lambda_1(\Sigma)^{d_1/2}},
\end{align*}
showing that 
\begin{align*}
  \sup_{y \in \R^{d_1}, y_2 \in \R^{d_2}, z_1 \in \R^{d_1}}\phi \left(y_1 - z_1 - \Sigma_{12}\Sigma_{22}^{-1}(y_2 - z_2), \Sigma_{11\cdot 2}\right) \leq \frac{1}{(2\pi \lambda_1(\Sigma))^{d_1/2}}.
\end{align*}
Thus, it follows that, for any $(y_1, y_2) \in \R^d$,
\begin{align*}
    p(y_1 \mid y_2) \leq \frac{1}{(2\pi \lambda_1(\Sigma))^{d_1/2}} \int \frac{\phi(y_2 - z_2, \Sigma_{22})}{p^{(2)}(y_2)} dF(z) =\frac{1}{(2\pi \lambda_1(\Sigma))^{d_1/2}},
\end{align*}
proving \eqref{eq_conditionalbound}.

\end{proof}

\clearpage
{\small
\bibliographystyle{apalike}
\bibliography{reference}
}


\end{document}